\documentclass{amsart}%
\usepackage{amsfonts}
\usepackage[all,cmtip]{xy}
\usepackage{amsmath}
\usepackage{amssymb}
\usepackage{graphicx}
\usepackage{hyperref}
\usepackage[margin=1.55in]{geometry}%
\newtheorem{theorem}{Theorem}
\theoremstyle{plain}
\newtheorem*{acknowledgement}{Acknowledgement}

\newtheorem{corollary}[theorem]{Corollary}
\newtheorem*{ncorollary}{Corollary}

\newtheorem{definition}[theorem]{Definition}

\newtheorem{lemma}[theorem]{Lemma}

\newtheorem{proposition}[theorem]{Proposition}
\theoremstyle{remark}
\newtheorem{example}[theorem]{Example}
\newtheorem{remark}[theorem]{Remark}

\numberwithin{equation}{section}
\numberwithin{theorem}{section}
\begin{document}
\title[Hilbert reciprocity using localization]{Hilbert reciprocity using $K$-theory localization}
\author{Oliver Braunling}
\email{oliver.braunling@gmail.com}
\thanks{The author was supported by the EPSRC Programme Grant EP/M024830/1
\textquotedblleft Symmetries and correspondences: intra-disciplinary
developments and applications\textquotedblright.}
\subjclass[2020]{ Primary 11A15, 11S70, Secondary 19C20}
\keywords{Hilbert reciprocity law, Moore sequence, localization sequence. Hilbert
symbol, tame symbol}

\begin{abstract}
Usually the boundary map in $K$-theory localization only gives the tame symbol
at $K_{2}$. It sees the tamely ramified part of the Hilbert symbol, but no
wild ramification. Gillet has shown how to prove Weil reciprocity using such
boundary maps. This implies Hilbert reciprocity for curves over finite fields.
However, phrasing Hilbert reciprocity for number fields in a similar way fails
because it crucially hinges on wild ramification effects. We resolve this
issue, except at $p=2$. Our idea is to pinch singularities near the
ramification locus. This fattens up $K$-theory and makes the wild symbol
visible as a boundary map.

\end{abstract}
\maketitle

\section{Introduction}

In the discipline of $K$-theory the term \emph{tame symbol} is a special name
for the boundary map $\partial_{v}\colon K_{2}(F)\rightarrow\kappa(v)^{\times
}$ for a field $F$ with a discrete valuation $v$ which has residue field
$\kappa(v)$. This map comes from the localization sequence, but it is also
very easy to give an explicit formula:%
\begin{equation}
\partial_{v}\{f,g\}:=(-1)^{v(f)v(g)}\overline{f^{v(g)}g^{-v(f)}}\in
\kappa(x)^{\times}\text{,} \label{lee1}%
\end{equation}
where $\overline{(-)}$ is the reduction modulo the valuation ideal.

In number theory the meaning of the term tame symbol is more restrictive.
Every finite extension $F/\mathbb{Q}_{p}$ with its usual valuation $v$ and
residue field $\kappa(v)$ is equipped, through a construction coming from
local class field theory, with its Hilbert symbol%
\[
h_{v}\colon K_{2}(F)\longrightarrow\mu(F)\text{,}%
\]
where $\mu(F)$ is the group of roots of unity. We shall recall its
construction and further background in \S \ref{sect_RecallHilbertSymbol}. The
residue field satisfies $\kappa(v)^{\times}\cong\mathbb{F}_{q}^{\times}$ and
using the Teichm\"{u}ller lift these roots of unity lift injectively to
elements of $\mu(F)$. The image of this lift agrees precisely with the direct
summand of those roots in $\mu(F)$ of order prime to $p$. Consider only the
output of $h_{v}$ on this summand,%
\[
K_{2}(F)\longrightarrow\mu(F)[\text{prime-to-}p\text{-part}]\cong%
\kappa(v)^{\times}\text{.}%
\]
Then this map agrees up to an isomorphism with the tame symbol. The
complementary $p$-torsion part $\mu(F)[p^{\infty}]$ has an interpretation in
terms of wild ramification. Seen from this viewpoint, the tame symbol takes
its name from encoding the tamely ramified information in the\ Hilbert symbol.
Unlike the tame symbol, the full Hilbert symbol has no simple formula as in
Equation \ref{lee1}. Formulas for the $p$-part of $h_{v}$ are known as
\emph{explicit reciprocity laws} and this is a whole area of research in
itself. Many explicit formulas exist, of varying generality, and all of them
are very complicated (as a very incomplete list, we name Artin--Hasse
\cite{MR3069494},\ Shafarevich \cite{MR0031944}, Wiles \cite{MR480442},
Br\"{u}ckner \cite{MR533354}, Vostokov \cite{MR522940}, Henniart
\cite{MR636453}, de Shalit \cite{MR835803},\ldots\footnote{There are many more
works and I apologize to those I\ have not explicitly listed.}).

Taking this into account, it is almost surprising that the localization
sequence in $K$-theory, a result which works in broad generality and has
nothing to do with the specifics of number theory, captures a part of the
Hilbert symbol at all.

It has been a long-standing folklore question whether it would also be
possible to express the full Hilbert symbol, including the wild part, as a
boundary map in a localization sequence. We answer this affirmatively in this
paper, except at places over the prime $p=2$ and real places. In other words:
We cannot capture $2$-torsion phenomena with what we do.

Whenever we work with a number field $F$, we have the instinct to work with
its ring of integers $\mathcal{O}_{F}$ as an integral model. Recall that an
\emph{order} $\mathcal{R}$ in $F$ is any subring $\mathcal{R}\subseteq F$ such
that $\mathbb{Q}\cdot\mathcal{R}=F$ and which is finitely generated as an
abelian group. The ring of integers $\mathcal{O}_{F}$ agrees with the unique
maximal order in $F$. This intuition is informed from ring-theoretic
properties: The maximal order is regular, it is our only chance to get a
Dedekind domain, to get unique factorization in prime ideals. However,
addressing our problem, it is best to shrink the order. First, proceed locally.

\begin{theorem}
\label{thm_Intro_1}For any finite extension $F/\mathbb{Q}_{p}$ with $p$ odd,
there exists a well-defined local subring $(\mathcal{R},\mathfrak{\tilde{m}})$
of the valuation ring such that the localization sequence for the closed-open
complement%
\[
\operatorname*{Spec}\mathcal{R}/\mathfrak{\tilde{m}}\hookrightarrow
\operatorname*{Spec}\mathcal{R}\hookleftarrow\operatorname*{Spec}F
\]
outputs%
\[
K_{2}(\mathcal{R})\longrightarrow K_{2}(F)\overset{\partial}{\longrightarrow
}K_{\mathfrak{\tilde{m}},1}(\mathcal{R})\longrightarrow1
\]
and such that the boundary map is the full Hilbert symbol. In particular,
$K_{\mathfrak{\tilde{m}},1}(\mathcal{R})\cong\mu(F)$. If $F/\mathbb{Q}_{p}$ is
unramified, $\mathcal{R}$ is the usual ring of integers.
\end{theorem}

See Theorem \ref{thm_LocalMain}. The ring $\mathcal{R}$, which we shall call
the \emph{optimal order}, will usually be singular. We explain its
construction later in the paper. Recall that $K_{\mathfrak{m},1}%
(\mathcal{O}_{F})$ gives only the prime-to-$p$ order roots of unity when using
the valuation ring $\mathcal{O}_{F}$ instead of $\mathcal{R}$. The
localization sequence in $K$-theory is compatible under switching from working
locally to globally. Running the same idea globally, one gets the full Hilbert
reciprocity law $-$ except that since the above only works for odd $p$, we
must invert $2$. This also washes away the contribution from all real places.

\begin{theorem}
\label{thm_Intro_2}Suppose $F$ is any number field. Then we construct a
well-defined order $\mathcal{R}\subseteq\mathcal{O}_{F}$ such that we get a
commutative diagram with exact rows%
\[%
\xymatrix{
K_{2}(\mathcal{R}) \ar[r] & K_{2}(F) \ar[r]^-{\partial} \ar@{=}[d] & \bigoplus
_{v \enspace\operatorname{finite}}K_{\mathfrak{\tilde{m}},1}(\mathcal{R}%
_{v}) \ar[d]^{\oplus_v \phi}
\ar[r] & SK_{1}(\mathcal{R}) \ar[r] \ar[d] & 0 \\
& K_{2}(F) \ar[r]_-{h_v} & \bigoplus_{{v \enspace\operatorname{noncomplex}}}
\mu(F_v) \ar[r]_-{\cdot\frac{m_v}{m}}
& \mu(F) \ar[r] & 0
}%
\]
in the category of abelian groups up to $2$-primary torsion. The top row is
the localization sequence with respect to the maximal ideals of $\mathcal{R}$.
The bottom row is the Moore sequence (a way to phrase Hilbert reciprocity).
The downward arrows are isomorphisms (up to $2$-primary torsion).
\end{theorem}

See Theorem \ref{thm_GlobalMain}, where we also elaborate on the formulation
\textquotedblleft up to $2$-primary torsion\textquotedblright. One can deduce
statements from this with a more classical flavour.

\begin{ncorollary}
Let $F$ be a number field with $\sqrt{-1}\notin F$. Then the statement of
Hilbert reciprocity up to sign, i.e.%
\[
\prod_{v\text{ }\operatorname*{noncomplex}}h_{v}(\alpha,\beta)^{\frac{m_{v}%
}{m}}=\pm1\qquad\text{for all}\qquad\alpha,\beta\in F^{\times}\text{,}%
\]
can be phrased as the property of being a complex ($d^{2}=0$) for a
localization sequence in $K$-theory.
\end{ncorollary}

As we will explain in \S \ref{sect_GilletsProof}, Gillet proved Hilbert
reciprocity for function fields over finite fields using $K$-theory
localization by proving the more general Weil reciprocity theorem over all
base fields. Under the number field/function field dictionary one would hope
that this proof generalizes to give Hilbert reciprocity for number fields.
This fails for various reasons, one being that one needs the \textit{full}
Hilbert symbol and not just the tame symbol. Our method solves this. However,
the version of Hilbert reciprocity it proves $-$ if we only use $K$-theory
localization and nothing else $-$ then takes values in the group $SK_{1}$ of
the global (singular) order we refer to in Theorem \ref{thm_Intro_2}.

It seems difficult to compute this group without using tools which would also
go into conventional proofs of Hilbert reciprocity.

It is a common phenomenon in Algebraic $K$-theory that adding nil-thickenings
or singularities makes $K$-groups more complicated. For example, the
Contou-Carr\`{e}re symbol is a fattening up of the tame symbol for
nil-thickenings and in its tangent functor one can find the residue symbol for
log differential forms. The residue symbol only becomes visible thanks to
nil-thickenings. The Contou-Carr\`{e}re symbol is a different kind of
generalization of the tame symbol than the Hilbert symbol. Nonetheless, our
main idea is similar. If we think in terms of the function field analogy, the
orders $\mathcal{R}$ correspond to (affine models of) singular algebraic
curves and $\operatorname*{Spec}\mathcal{O}_{F}\rightarrow\operatorname*{Spec}%
\mathcal{R}$ is the normalization, the resolution of singularities. Using this
analogy, working with $\mathcal{R}$ is like pinching additional singularities
at the branching locus of $F$ over $\mathbb{Q}$. This then has the effect to
fatten up the $K$-theory with support at these singularities. As it turns out,
this additional complexity is precisely the wild part of the Hilbert symbol.
Under the normalization, going to $\mathcal{O}_{F}$, $K$-theory loses the wild
data and keeps only the tame symbol in the boundary map.

Digging a littler deeper, there is a hierarchy of rings $\mathcal{R}_{i}$
between the optimal and the maximal order,%
\[
\mathcal{R}\subseteq\mathcal{R}_{i}\subseteq\mathcal{O}_{F}%
\]
which \textquotedblleft can see\textquotedblright, as $i$ increases, less and
less of the wild part of the Hilbert symbol. We shall study these in more
detail in a future text.

Considering local metaplectic extensions, Toshiaki Suzuki has worked with
rings similar to what we do in \S \ref{sect_LocalTheory} with the purpose to
formulate modified Gauss sums \cite{MR1958068,MR2293361,MR3052658}. He does
not use $K$-theory, but probably our local optimal order agrees with his
subrings from \cite{MR1958068}.

Let me note that the recent papers \cite{MR3954369, clausennc} also produce,
without relying on global class field theory, a sequence%
\[
K_{2}(F)\overset{\partial}{\longrightarrow}\bigoplus_{v}\mu(F_{v}%
)\longrightarrow K_{2}(\mathsf{LCA}_{F})_{/div}\longrightarrow0\text{,}%
\]
which is then shown to be equivalent to Moore's sequence. The methods in these
papers are completely different from what we do in this paper.

\begin{acknowledgement}
I have learned about this problem mostly from Ivan Fesenko and Mikhail
Kapranov, to whom I\ am very grateful. Recently, renewed interest in such
questions was sparked by the fresh summer breeze of inspiration coming from
the writings of Dustin Clausen \cite{clausenJ, clausen}, even though they go
into a different direction from what we are doing in this paper. I thank D.
Clausen, M. Groechenig and M. Tamiozzo for their
feedback and various corrections.
\end{acknowledgement}

\section{What is the Hilbert symbol?}

\subsection{Local theory\label{sect_RecallHilbertSymbol}}

In this section we shall recall the definition and purpose of the Hilbert
symbol and the statement of the Hilbert reciprocity law. Readers familiar with
these concepts are invited to skip ahead.

Suppose $F$ is a local field of characteristic zero, i.e. a finite extension
of $\mathbb{Q}_{p}$ or $\mathbb{R}$. In this section, let us exclude the case
$F=\mathbb{C}$. Write $m:=\#\mu(F)<\infty$ for the number of roots of unity in
$F$.\footnote{We have excluded $\mathbb{C}$ because it is the only local field
with infinitely many roots of unity. However, it also has trivial Hilbert
symbol, so we do not miss out on anything by skipping this case.} Then
the\emph{ local reciprocity map} (also known as \emph{local Artin map}) is a
group homomorphism%
\[
\operatorname*{Art}\nolimits_{F}\colon F^{\times}\longrightarrow
\operatorname*{Gal}(F^{\operatorname{ab}}/F)\text{,}%
\]
where $F^{\operatorname{ab}}/F$ is the maximal abelian extension. It is easy
to describe for $F=\mathbb{Q}_{p}$ or $\mathbb{R}$, but for a general local
field setting up $\operatorname*{Art}\nolimits_{F}$ is too complicated than we
would wish to recall here. The Artin map becomes an isomorphism after
profinite completion\footnote{The Galois group is of course already profinite.
However, the left side changes. For $p$-adic fields the completion injects
into a strictly bigger group, while for the reals the completion map is
surjective. The kernel is the subgroup of positive reals.}, so it gives tight
control over the abelian field extensions of $F$. The \emph{Hilbert symbol} is
the map%
\begin{align}
h\colon F^{\times}\times F^{\times}  &  \longrightarrow\mu(F)\nonumber\\
(x,y)  &  \longmapsto\frac{\operatorname*{Art}_{F}(x)\sqrt[m]{y}}{\sqrt[m]{y}%
}\text{.} \label{lcimde1}%
\end{align}
We need to explain this: Since $F$ contains a primitive $m$-th root of unity,
the extension $F(\sqrt[m]{y})/F$ is Kummer, and in particular an abelian
extension. Only because of this, $\operatorname*{Art}_{F}(x)$ acts on it. All
Galois conjugates of $\sqrt[m]{y}$ differ by a root of unity, so the map takes
values in $\mu(F)$. It is well-defined because if we pick a different root for
$\sqrt[m]{y}$, then (as long as we use the same one in the numerator and
denominator of Equation \ref{lcimde1}) they both differ by the same root of
unity from our previous choices, which then cancels out in the fraction.

One checks that $h$ is bilinear and respects the Steinberg relation, so we
obtain a factorization%
\[
h\colon K_{2}(F)\longrightarrow\mu(F)\text{.}%
\]
In a nutshell, one could say that $h$ sets up a link between the local
reciprocity map and Kummer theory. Suppose $F$ is a finite extension of
$\mathbb{Q}_{p}$. As $\mu(F)$ is a torsion group, it splits canonically into
its $p$-primary part and a prime-to-$p$ part,%
\[
\mu(F)\cong\mu(F)[p^{\infty}]\oplus%
{\textstyle\bigoplus_{\ell\neq p}}
\mu(F)[\ell^{\infty}]\text{.}%
\]
We have no use for further distinguishing the primes in the prime-to-$p$
summand, we just write $\mu(F)[$prime-to-$p]$ for it all.

Explicit formulas for the $p$-part of the Hilbert symbol, i.e. its values in
the summand $\mu(F)[p^{\infty}]$, are very complicated. If we write
$m=p^{k}m_{0}$ with $(p,m_{0})=1$, then $F(\sqrt[m]{y^{p^{k}}})=F(\sqrt[m_{0}%
]{y})$ is at worst tamely ramified over $F$, and in this case the value of the
Hilbert symbol $(x,y^{p^{k}})$ in the summand $\mu(F)[p^{\infty}]$ is trivial.
Hence, such a case corresponds to (at worst) tame ramification. In this
situation, the Hilbert symbol agrees wth the tame symbol, see Lemma
\ref{lemma_HilbertBecomesTameSymbol} below. The values in the complementary
summand $\mu(F)[p^{\infty}]$ are therefore also sometimes called the wild part
of the Hilbert symbol, or even a \emph{wild symbol} formula.

\subsection{Global theory}

Suppose $F$ is a number field. For any noncomplex place write $F_{v}$ for its
completion at $v$. Define $m:=\#\mu(F)<\infty$ and $m_{v}:=\#\mu(F_{v}%
)<\infty$.

\begin{theorem}
[Moore sequence]\label{thm_MooreSequence}Let $F$ be a number field. Then the
sequence%
\begin{equation}
K_{2}(F)\longrightarrow\bigoplus_{v\text{ }\operatorname*{noncomplex}}%
\mu(F_{v})\overset{\cdot\frac{m_{v}}{m}}{\longrightarrow}\mu(F)\longrightarrow
0 \label{ldiva1}%
\end{equation}
is exact, where the first arrow is the Hilbert symbol, and the second is
taking the $\frac{m_{v}}{m}$-th power.
\end{theorem}

The statement that this is a complex ($d^{2}=0$) amounts to the classical
Hilbert reciprocity law. Moore provided the slight strengthening that the
sequence is exact. We refer to \cite{MR311623} for a nice proof.

\begin{example}
\label{example_QuadraticReciprocity}For $F=\mathbb{Q}$ we have $\mu
(F):=\{\pm1\}$. The finite places are just the ordinary prime numbers. If $p$
is odd, $\frac{m_{p}}{m}=\frac{p-1}{2}$ is half the order of $\mathbb{F}%
_{p}^{\times}$ and coprime to $p$, so the Hilbert symbol is just the tame
symbol, but effectively only contributes values in $\mathbb{F}_{p}^{\times
}/\mathbb{F}_{p}^{\times2}\cong\{\pm1\}$ thanks to taking the $\frac{m_{p}}%
{m}$-th power. One finds that this composition $K_{2}(F)\rightarrow\{\pm1\}$
is given in terms of the Legendre symbol for all odd $p$. Only for
$p=2,\infty$ the Hilbert symbol differs from the tame symbol. The fact that
$d^{2}=0$ in Equation \ref{ldiva1} then becomes the Gauss quadratic
reciprocity law.
\end{example}

\section{Gillet's proof of Weil reciprocity\label{sect_GilletsProof}}

To motivate this paper, we must recall Gillet's very elegant proof of Hilbert
reciprocity for curves over finite fields. This is the counterpart, under the
number field/function field dictionary, of the Hilbert reciprocity law which
we discuss in the present article. Gillet first proves something more general.

\begin{theorem}
[Weil reciprocity]Let $k$ be a field. Let $X/k$ be an integral
smooth\footnote{Gillet's proof actually does not require smoothness if one
works with $G$-theory localization rather than $K$-theory. But discussing this
detour would distract us too much from our main story. We favour $K$-theory
for good reasons over $G$-theory in this paper.} proper curve with function
field $F:=k(X)$. Then the composition
\[
K_{n}(F)\overset{\partial}{\longrightarrow}\bigoplus_{v}K_{n-1}(\kappa
(v))\overset{N_{\kappa(v)/k}}{\longrightarrow}K_{n-1}(k)
\]
is zero, where $\partial$ denotes the boundary map of the Algebraic $K$-theory
localization sequence with respect to zero-dimensional support, $v$ runs
through the places of $X$ which are trivial over $k$, and $N_{\kappa(v)/k}$ is
the norm map.
\end{theorem}

The set of places $v$ can be identified with the set of codimension one
irreducible closed subschemes, and equivalently with the set of closed points
of $X$. The local rings $(\mathcal{O}_{X,P},\mathfrak{m}_{P})$ are DVRs and
the resulting valuation defines a place $v$ on the function field, and
conversely these valuations pin down a local ring and therefore a closed point.

\begin{proof}
[Proof (Gillet)]We shall freely use results from Algebraic $K$-theory. The
reader will find in Appendix \S \ref{sect_appendix_Kthy} a summary, in
blackbox format, of what we need. Following Gillet, we use the localization
sequence of Algebraic $K$-theory applied to $X$ and $Z$ being any
zero-dimensional integral closed subscheme. We get the fiber sequence%
\begin{equation}
K_{Z}(X)\longrightarrow K(X)\longrightarrow K(X-Z)\text{.} \label{lcff1}%
\end{equation}
Note that we may split $Z$ into its disjoint connected components and get
$K_{Z}(X)=\bigoplus_{P}K_{P}(X)$, where $P$ runs through the finitely many
points of $Z$. Finally, since $X$ is smooth, it is regular, and by devissage
we obtain $K_{P}(X)\cong K(P)$ (i.e. regard $P$ as a zero-dimensional closed
subscheme). Note that this means that $K(P)=K(\kappa(v))$, where
$\kappa(v)=\mathcal{O}_{X,P}/\mathfrak{m}_{P}$ is the residue field of the DVR
belonging to the valuation $v$. Next, we arrange this in a direct system
running over $Z\subseteq Z^{\prime}$, partially ordered under containment. As
$K$-theory commutes with filtering colimits, we obtain the top row in%
\begin{equation}%
\xymatrix{
\cdots\ar[r] & K_n(F) \ar[dr] \ar[r]^-{\partial} & \bigoplus_v K_{n-1}%
(\kappa(v)) \ar[d]^{N_{\kappa(v)/k}} \ar[r] & K_{n-1}(X) \ar[r] \ar
[dl]^{R\pi_{\ast}} & \cdots\\
&                            & K_{n-1}(k) &
}
\label{lcffA1}%
\end{equation}
Since the structure map $\pi:X\rightarrow k$ is proper, the derived
pushforward $R\pi_{\ast}$ induces a map $K(X)\rightarrow K(k)$. Since the
closed embeddings $i_{P}\colon P\hookrightarrow X$ are also proper, $\pi\circ
i_{P}$ induces maps $K(P)\rightarrow K(k)$. Since the composition of two
consecutive arrows in a complex is zero ($d^{2}=0$), the existence of the
right diagonal arrow implies that $\sum_{P}\pi_{\ast}i_{P\ast}=0$. Upon
unravelling this, the $\pi\circ i_{P}$ correspond to the pushforward along the
finite morphisms $\operatorname*{Spec}\kappa(v)\rightarrow\operatorname*{Spec}%
k$, which is nothing but the norm map on $K$-theory.
\end{proof}

The above proof, which uses only fairly general tools from $K$-theory and the
existence of $R\pi_{\ast}$ now gives a complete proof of Hilbert reciprocity
for curves over finite fields.

\begin{theorem}
[Hilbert reciprocity for function fields]\label{thm_WeilForNTwoImpliesHilbert}%
Suppose $X/\mathbb{F}_{q}$ is a geometrically integral smooth proper curve
with function field $F:=\mathbb{F}_{q}(X)$. Then Hilbert reciprocity
holds:\ The composition of%
\begin{equation}
K_{2}(F)\overset{h_{v}}{\longrightarrow}\bigoplus_{v}\mu(F_{v})\overset
{\cdot\frac{m_{v}}{m}}{\longrightarrow}\mu(F) \label{lhh2}%
\end{equation}
is zero, where $m_{v}:=\#\mu(F_{v})$ and $m:=\mu(F)$.
\end{theorem}

We have only discussed the Hilbert symbol in characteristic zero in
\S \ref{sect_RecallHilbertSymbol}, but the same definitions go through in
general. We will not work in positive characteristic anywhere else in this paper.

\begin{proof}
Write $k:=\mathbb{F}_{q}$ for the base field. We use Weil reciprocity for
$n=2$ and obtain%
\begin{equation}
K_{2}(F)\overset{\partial}{\longrightarrow}\bigoplus_{v}K_{1}(\kappa
(v))\overset{N_{\kappa(v)/k}}{\longrightarrow}K_{1}(k)\text{.} \label{lhh1}%
\end{equation}
Since $F$ is a function field of a geometrically integral curve, we have the
sequence $k^{\times}\hookrightarrow F^{\times}\overset{v}{\rightarrow
}\bigoplus\mathbb{Z}$ sending a function to its Weil divisor, and since
$\mathbb{Z}$ has no non-trivial torsion, it follows that all roots of unity
must come from $k^{\times}$, but since $k=\mathbb{F}_{q}$, all its non-zero
elements are roots of unity. Hence, $K_{1}(k)=\mu(F)$. Similarly, each
$\kappa(v)/\mathbb{F}_{q}$ is a finite field extension, and thus $\kappa(v)$
is itself a finite field and therefore $K_{1}(\kappa(v))=\mu(F_{v})$. Note
that since $k$ has characteristic $p>0$ and we have $\mu(F_{v})=\kappa
(v)^{\times}\oplus\mu(F_{v})[p^{\infty}]$ by the Teichm\"{u}ller lift, we must
have $\mu(F_{v})[p^{\infty}]=1$ as there cannot be non-trivial $p$-power roots
of unity in a $\operatorname*{char}p>0$ field. This also implies that the
Hilbert symbol agrees with the tame symbol. By now, Equation \ref{lhh1} has
almost completely transformed into Equation \ref{lhh2}. It remains to show
that $N_{\kappa(v)/k}$ agrees with multiplication by $\frac{m_{v}}{m}$, but%
\[
\frac{m_{v}}{m}=\frac{q^{[\kappa(v):k]}-1}{q-1}=1+q+q^{2}+\cdots
+q^{[\kappa(v):k]-1}\text{,}%
\]
and the extension $\kappa(v)/k$ is Galois and its Galois group generated by
the Frobenius $x\mapsto x^{q}$. Thus, taking the $\frac{m_{v}}{m}$-th power
does the same as taking the norm.
\end{proof}

Inspired by this beautifully short proof, people have been wondering about the
possibility to prove Hilbert reciprocity for number fields using the same idea.

For the sake of completeness, let us quickly go through the same steps in the
number field setting and see how and why they fail.

Let $F$ be a number field and take $X:=\operatorname*{Spec}\mathcal{O}_{F}$,
where $\mathcal{O}_{F}$ denotes the ring of integers. The localization
sequence of Equation \ref{lcff1} works fine, where again $Z$ is running
through the direct system of zero-dimensional closed subschemes partially
ordered by inclusion. We obtain the exact sequence%
\[
\cdots\longrightarrow K_{n}(F)\overset{\partial}{\longrightarrow}\bigoplus
_{v}K_{n-1}(\kappa(v))\overset{\beta_{n}}{\longrightarrow}K_{n-1}%
(\mathcal{O}_{F})\longrightarrow\cdots\text{,}%
\]
where $v$ runs through the finite places. What should the replacement for
$\pi$ in Diagram \ref{lcffA1} be? We might speculate about $\pi
:\operatorname*{Spec}\mathcal{O}_{F}\rightarrow\operatorname*{Spec}\mathbb{Z}$
or about some map to a conjectural $\mathbb{F}_{1}$, but there is a problem we
cannot avoid: Soul\'{e} has proven that for all even $n\geq2$ the maps
$\beta_{n}$ are zero \cite[Chapter V, Theorem 6.8]{MR3076731}. So even if we
can set up some magical diagram (with $n\geq2$ even)%
\[%
\xymatrix{
\cdots\ar[r] & K_n(F) \ar[dr] \ar[r]^-{\partial} & \bigoplus_v K_{n-1}%
(\kappa(v)) \ar[d] \ar[r]^-{0} & K_{n-1}(\mathcal{O}_F) \ar[r] \ar@
{..>}[dl] & \cdots\\
&                            & K_{n-1}(?) &
}%
\]
for some carefully chosen \textquotedblleft$?$\textquotedblright, we see that
for $n=2$ the statement of our theorem is void because already the individual
maps $\pi_{\ast}\circ i_{P\ast}$ are necessarily zero. So clearly their sum is
zero. We learn nothing new.

Is this because we have no infinite places here and some Arakelov theory would
solve the problem? At first, one could say yes. The cited result of Soul\'{e}
also tells us that for affine curves over finite fields, e.g.
$\operatorname*{Spec}\mathcal{O}_{X}(X-pt)$, the same vanishing holds. But we
could take a number field without real places. Since Hilbert reciprocity among
the infinite places only sees the real ones, this leads us to situations where
Hilbert reciprocity is exclusively a statement about a cancellation of values
at \textit{finite} places, but we still have the same issue. So this is not
the core of the problem.

More crucially, we only see the \textit{tame} symbol here. Consider for
example the extensions $\mathbb{Q}_{p}(\zeta_{p^{m}})/\mathbb{Q}_{p}$. These
are totally ramified. As $m\rightarrow+\infty$, the group $\mu(\mathbb{Q}%
_{p}(\zeta_{p^{m}}))=\left\langle \zeta_{p^{m}}\right\rangle \times
\mathbb{F}_{q}^{\times}$ grows arbitrarily large, yet the receptacle of the
tame symbol%
\[
K_{2}(\mathbb{Q}_{p}(\zeta_{p^{m}}))\overset{\partial}{\longrightarrow}%
K_{1}(\mathbb{F}_{p})=\mathbb{F}_{p}^{\times}%
\]
remains unchanged. This shows how the discrepancy between the Hilbert symbol
and the tame symbol can grow arbitrarily large.

\begin{example}
[continues Example \ref{example_QuadraticReciprocity}]Another good example is
the quadratic reciprocity law at $p=2$. The tame symbol here is a map
$K_{2}(\mathbb{Q}_{2})\rightarrow\mathbb{F}_{2}^{\times}=\{1\}$, and contains
no information. The Hilbert symbol of $\mathbb{Q}_{2}$ however is non-trivial.
\end{example}

Even once we know that the Hilbert reciprocity law holds, it does not really
imply anything about a reciprocity law for tame symbols, so it is not too
shocking that Soul\'{e}'s result reveals that considering the tame symbol
yields no interesting invariant whatsoever.

\section{Local theory\label{sect_LocalTheory}}

Let $F/\mathbb{Q}_{p}$ be a finite extension. We write $F_{0}/\mathbb{Q}_{p}$
for its maximal unramified subextension. Denote by $(\mathcal{O}%
_{F},\mathfrak{m})$ the local ring which is the ring of integers of $F$.
Analogously, write $(\mathcal{O}_{0},p)$ for the ring of integers of $F_{0}$.

We recall that we may realize $\mathcal{O}_{F}=\mathcal{O}_{0}[T]/(f)$, where
$f$ is an Eisenstein polynomial for the prime $p$, \cite[(5.6)]{MR1972204}.
The image of the variable, $\pi:=\overline{T}$, is a uniformizer of $F$, so
that $\mathfrak{m}=(\pi)\mathcal{O}_{F}$. Conversely, the minimal polynomial
of any uniformizer of $F$ over $\mathcal{O}_{0}$ will be a possible choice for
the Eisenstein polynomial $f$, \cite[(5.12)]{MR1972204}.

As a result of this presentation, $\mathcal{O}_{F}$ is a finitely generated
free $\mathcal{O}_{0}$-module, and we may take%
\begin{equation}
\mathcal{O}_{F}=\mathcal{O}_{0}\left\langle 1,\pi,\ldots,\pi^{e-1}%
\right\rangle \label{lc1}%
\end{equation}
as a basis, where $e:=\deg f$ is the ramification index of $F/\mathbb{Q}_{p}$.
Moreover, recall that there is a non-canonical decomposition%
\begin{equation}
F^{\times}\cong\left\langle \pi^{\mathbb{Z}}\right\rangle \times\left\langle
\zeta_{q-1}\right\rangle \times U_{F}^{1}\text{,} \label{lc1a}%
\end{equation}
where $\zeta_{q-1}$ is a primitive $(q-1)$-st root of unity, and%
\[
U_{F}^{r}:=1+\mathfrak{m}^{r}\mathcal{O}_{F}\qquad\text{(for }r\geq1\text{)}%
\]
denotes the $r$-th higher unit group.

Let us consider the subrings\footnote{This is a special case of a general
observation: If $R\subseteq S$ is a subring of any ring $S$, and $I$ any ideal
of $S$, then $R+I$ is a ring.}%
\begin{equation}
R_{m}:=\mathcal{O}_{0}+\mathfrak{m}^{m}\mathcal{O}_{F}\qquad\qquad\text{(for
}m\geq0\text{)} \label{lcqx1}%
\end{equation}
inside $\mathcal{O}_{F}$.

\begin{lemma}
\label{lemma_Rm_main}The ring $R_{m}$ is a Noetherian local domain with field
of fractions $F$. Its maximal ideal is%
\begin{equation}
\mathfrak{\tilde{m}}_{m}=\left\{
\begin{array}
[c]{ll}%
p\mathcal{O}_{0}+\mathfrak{m}^{m}\mathcal{O}_{F} & \text{for }m\geq1\\
\mathfrak{m}\mathcal{O}_{F} & \text{for }m=0\text{.}%
\end{array}
\right.  \label{lc1c}%
\end{equation}
The ring $R_{m}$ is a module-finite $\mathcal{O}_{0}$-algebra of finite index
$[\mathcal{O}_{F}:R_{m}]<\infty$. Moreover, $R_{m}$ is a compact subring of
$\mathcal{O}_{F}$ (in the valuation topology). The unit group is
non-canonically isomorphic to%
\begin{equation}
R_{m}^{\times}\cong\left\langle \zeta_{q-1}\right\rangle \times
(1+\mathfrak{\tilde{m}}_{m})\text{.} \label{lc1aa}%
\end{equation}
If $1\leq m\leq e$ (with $e:=e_{F/\mathbb{Q}_{p}}$ the absolute ramification
index), $\mathfrak{\tilde{m}}_{m}=\mathfrak{m}^{m}\mathcal{O}_{F}$.
\end{lemma}

\begin{proof}
Since $R_{0}=\mathcal{O}_{F}$, the case $m=0$ is clear. From now on assume
$m\geq1$. It is clear that the $R_{m}=\mathcal{O}_{0}+\mathfrak{m}%
^{m}\mathcal{O}_{F}$ are subrings of $\mathcal{O}_{F}$. It follows that
$R_{m}$ is a domain. Next, it is immediate to check that $\mathfrak{\tilde{m}%
}_{m}\subseteq R_{m}$ is an ideal. Since $\mathfrak{\tilde{m}}_{m}%
\subseteq\mathfrak{m}$, it does not contain $1$, so it is a proper ideal in
$R_{m}$. If we can show that all elements in $R_{m}\setminus\mathfrak{\tilde
{m}}_{m}$ are units, it will follow that $\mathfrak{\tilde{m}}_{m}$ is the
unique maximal ideal. Before we do this, note that the decomposition of
Equation \ref{lc1a} applied to $F_{0}$ shows that $\zeta_{q-1}\in
\mathcal{O}_{0}$ (since both $F$ and $F_{0}$ have the same residue field;
alternatively retrieve $\zeta_{q-1}$ as the Teichm\"{u}ller lift of a
primitive $(q-1)$-st root of unity of the finite residue field). Suppose $x\in
R_{m}\setminus\mathfrak{\tilde{m}}_{m}$. Using Equation \ref{lc1a} for $F$ we
write it as%
\begin{equation}
x=\pi^{n}\cdot\zeta_{q-1}^{i}\cdot u\qquad\text{with}\qquad u\in U_{F}%
^{1}\text{.} \label{lc1b}%
\end{equation}
Since $x\in\mathcal{O}_{F}$, we have $n\geq0$. Assume $n\geq1$. Then%
\[
x\in R_{m}\cap\mathfrak{m}\mathcal{O}_{F}=(\mathcal{O}_{0}+\mathfrak{m}%
^{m}\mathcal{O}_{F})\cap\mathfrak{m}\mathcal{O}_{F}=\left(  \mathcal{O}%
_{0}\cap\mathfrak{m}\mathcal{O}_{F}\right)  +\mathfrak{m}^{m}\mathcal{O}%
_{F}=\mathfrak{\tilde{m}}_{m}%
\]
since $m\geq1$ and $\mathcal{O}_{0}\cap\mathfrak{m}\mathcal{O}_{F}%
=p\mathcal{O}_{0}$. Contradiction. Hence, we must have $n=0$. As $\zeta
_{q-1}\in\mathcal{O}_{0}$, we deduce%
\begin{equation}
x\zeta_{q-1}^{-i}-1\in\mathfrak{m}\mathcal{O}_{F}\cap R_{m} \label{lc1bv}%
\end{equation}
(it lies in $\mathfrak{m}\mathcal{O}_{F}$ by Equation \ref{lc1b}, and in
$R_{m}$ since $x$ and $\zeta_{q-1}^{-1}$ lie in $R_{m}$ and $R_{m}$ is closed
under ring operations). Then the geometric series%
\begin{equation}
\frac{1}{x\zeta_{q-1}^{-i}}=\frac{1}{1-\left(  1-x\zeta_{q-1}^{-i}\right)
}=1+\sum_{l\geq1}\left(  1-x\zeta_{q-1}^{-i}\right)  ^{l}:=\tilde{u}
\label{lc1bb}%
\end{equation}
is convergent in $\mathcal{O}_{F}$. As the residue field of $\mathcal{O}_{F}$
is finite, $\mathcal{O}_{F}$ is compact in the valuation topology. Same for
$\mathcal{O}_{0}$. As $R_{m}$ is the image of the compact space $\mathcal{O}%
_{0}\times\mathcal{O}_{F}$ under an obvious map, $R_{m}$ is compact. In
particular, $R_{m}$ is necessarily complete as a metric space. As all partial
sums in Equation \ref{lc1bb} lie in $R_{m}$, we deduce that the limit
$\tilde{u}$ must also lie in $R_{m}$. Returning to our original thread of
thoughts, Equation \ref{lc1bb} shows that $x^{-1}=\zeta_{q-1}^{-i}\tilde{u}$
with $\tilde{u}\in R_{m}$.\ This completes the proof that $\mathfrak{\tilde
{m}}_{m}$ is the unique maximal ideal. We conclude that $R_{m}$ is a local
ring. As this also implies that $R_{m}^{\times}=R_{m}\setminus\mathfrak{\tilde
{m}}_{m}$, the computation starting from $x\in R_{m}\setminus\mathfrak{\tilde
{m}}_{m}$ and leading to Equation \ref{lc1bv} proves the inclusion
$R_{m}^{\times}\subseteq\left\langle \zeta_{q-1}\right\rangle \times
(1+\mathfrak{\tilde{m}}_{m})$ of Equation \ref{lc1aa}; and Equation
\ref{lc1bb} had already shown the reverse inclusion. Next, let us show that
$\operatorname*{Frac}R_{m}=F$. This is easy: Since $F=\operatorname*{Frac}%
\mathcal{O}_{F}$ is a discrete valuation field, for any $x\in F$ we have
$\pi^{l}x\in\mathcal{O}_{F}$ for $l\geq0$ sufficiently big. In particular,
$\pi^{l+m}x\in R_{m}$ and $\pi^{l+m}\in R_{m}$, so $x=\frac{\pi^{l+m}x}%
{\pi^{l+m}}\in\operatorname*{Frac}R_{m}$. Finally, by Equation \ref{lc1} we
obtain%
\[
R_{m}=\mathcal{O}_{0}+\pi^{m}\mathcal{O}_{0}\left\langle 1,\pi,\ldots
,\pi^{e-1}\right\rangle \text{,}%
\]
so $R_{m}$ is finitely generated as an $\mathcal{O}_{0}$-module. As
$\mathcal{O}_{0}$ is a discrete valuation ring, the structure theorem for
finitely generated modules over a principal ideal domain shows that $R_{m}$ is
a finite free $\mathcal{O}_{0}$-module. There cannot be any torsion since it
is a submodule of $F$, hence $p$-torsionfree. This also implies that $R_{m}$
is Noetherian (which could also be shown in many other ways). If $1\leq m\leq
e$, then $p\in\pi^{e}\mathcal{O}_{F}^{\times}\subseteq\pi^{m}\mathcal{O}_{F}$,
so $\mathfrak{\tilde{m}}_{m}=\mathfrak{m}^{m}\mathcal{O}_{F}$. As $R_{m}$ can
be sandwiched as $\mathfrak{m}^{m}\mathcal{O}_{F}\subseteq R_{m}%
\subseteq\mathcal{O}_{F}$ but the outer two terms have finite index in one
another, we also must have $[\mathcal{O}_{F}:R_{m}]<\infty$. This also proves
$\operatorname*{Frac}R_{m}=F$ a second time.
\end{proof}

\begin{lemma}
The residue field of $R_{m}$ is the same as of $\mathcal{O}_{F}$, i.e.
$\mathcal{O}_{F}/\mathfrak{m}$. For $m\geq1$ we have $R_{m}=\mathcal{O}%
_{0}+\mathfrak{\tilde{m}}_{m}$.
\end{lemma}

\begin{proof}
Suppose $m\geq1$. Then $R_{m}=\mathcal{O}_{0}+\mathfrak{\tilde{m}}_{m}$ by
Equation \ref{lc1c}. For subgroups $A,B$ contained in some abelian group one
generally has%
\[
\frac{A+B}{B}\cong\frac{A}{A\cap B}\text{.}%
\]
Hence, for $A:=\mathcal{O}_{0}$ and $B:=\mathfrak{\tilde{m}}_{m}$, both lying
inside $\mathcal{O}_{F}$, we find%
\[
\frac{R_{m}}{\mathfrak{\tilde{m}}_{m}}=\frac{\mathcal{O}_{0}+\mathfrak{\tilde
{m}}_{m}}{\mathfrak{\tilde{m}}_{m}}\cong\frac{\mathcal{O}_{0}}{\mathcal{O}%
_{0}\cap\mathfrak{\tilde{m}}_{m}}=\frac{\mathcal{O}_{0}}{\left(
\mathcal{O}_{0}\cap p\mathcal{O}_{0}\right)  +(\mathcal{O}_{0}\cap\pi
^{m}\mathcal{O}_{F})}\text{.}%
\]
Since $\pi\mathcal{O}_{F}\cap\mathcal{O}_{0}=p\mathcal{O}_{0}$, we have
$\mathcal{O}_{0}\cap\pi^{m}\mathcal{O}_{F}=p^{m}\mathcal{O}_{F}$, so%
\[
=\frac{\mathcal{O}_{0}}{p\mathcal{O}_{0}+p^{m}\mathcal{O}_{0}}=\frac
{\mathcal{O}_{0}}{p\mathcal{O}_{0}}\text{,}%
\]
which is the residue field of $\mathcal{O}_{0}$, but since $F/F_{0}$ is
totally ramified, this is the same residue field as of $F$. Finally, for $m=0$
the first claim is clearly also true.
\end{proof}

\begin{lemma}
\label{lemma_PrincipalUnitsZpModuleStructure}The group $1+\mathfrak{\tilde{m}%
}_{m}\subseteq R_{m}^{\times}$ is a $\mathbb{Z}_{p}$-module under
exponentiation. Concretely, one may define%
\[
(1+x)^{\alpha}:=\sum_{l\geq0}\dbinom{\alpha}{l}x^{l}%
\]
for all $\alpha\in\mathbb{Z}_{p}$, where $\dbinom{\alpha}{l}=\frac
{\alpha(\alpha-1)\cdots(\alpha-l+1)}{1\cdot2\cdots l}$.
\end{lemma}

\begin{proof}
This result is well-known for $U_{F}^{1}$, so it is clear that the series
converges to an element in $U_{F}^{1}$. As soon as we show that the limit lies
in $1+\mathfrak{\tilde{m}}_{m}$, it is then clear that this pairing
$\mathbb{Z}_{p}\times(1+\mathfrak{\tilde{m}}_{m})\rightarrow U_{F}^{1}$
restricts to a module structure on $1+\mathfrak{\tilde{m}}_{m}$. To see this,
first note that $\dbinom{\alpha}{l}\in\mathbb{Z}_{p}$ (because for all
$\alpha\in\mathbb{Z}$ the value lies in $\mathbb{Z}$, so by continuity for
$\alpha\in\mathbb{Z}_{p}$ the values must lie in the closure $\overline
{\mathbb{Z}}=\mathbb{Z}_{p}$ with respect to the $p$-adic topology), so once
$l\geq1$, we have $\dbinom{\alpha}{l}x^{l}\in\mathfrak{\tilde{m}}_{m}$ since
$x\in\mathfrak{\tilde{m}}_{m}$. As $\mathfrak{\tilde{m}}_{m}$ is a closed
subspace of $R_{m}$, we deduce that the limit lies in $1+\mathfrak{\tilde{m}%
}_{m}$. This concludes the proof. An alternative approach could be modelled on
\cite[Chapter I, (6.1)]{MR1915966}.
\end{proof}

\begin{lemma}
\label{lemma_RmIsClopenInF}$R_{m}$ is clopen in $F$ and the quotient group
$F/R_{m}$ is discrete.
\end{lemma}

Clear.

\begin{lemma}
\label{lemma_IsIntegralClosure}The ring $\mathcal{O}_{F}$ is the integral
closure of $R_{m}$ in $F$.
\end{lemma}

\begin{proof}
We use that $\mathcal{O}_{F}$ is the integral closure of $\mathcal{O}_{0}$ in
$F$ by \cite[(5.6), (i)]{MR1972204}. By the transitivity of integral closures,
it follows that the intermediate ring $\mathcal{O}_{0}\subseteq R_{m}%
\subseteq\mathcal{O}_{F}$ must have the same integral closure.
\end{proof}

Having shown that $R_{m}$ is a local ring, we may now apply the following
general result.

\begin{proposition}
[Dennis--Stein]\label{prop_DennisStein}Suppose $R$ is a local ring.

\begin{enumerate}
\item Then $K_{1}(R)\cong R^{\times}$ (via the determinant) and the product
map of $K$-theory,%
\[
R^{\times}\otimes R^{\times}\longrightarrow K_{2}(R)
\]
is surjective, i.e. every element in $K_{2}$ comes from symbols.

\item \emph{(Steinberg relation)} If $x,y\in R^{\times}$ such that $x+y=1$,
then $\{x,y\}=0$.
\end{enumerate}
\end{proposition}

\begin{proof}
Since $R$ is a local ring (Lemma \ref{lemma_Rm_main}), $R^{\times
}=\operatorname{GL}_{1}(R)\rightarrow K_{1}(R)$ is an isomorphism, and split
by the determinant $\operatorname{GL}(R)\rightarrow R^{\times}$. By a result
of Dennis--Stein \cite[Theorem 2.7]{MR0406998} for a local ring the map
$K_{1}(R)\otimes K_{1}(R)\rightarrow K_{2}(R)$ is surjective. The claim in (2)
is standard.
\end{proof}

\begin{lemma}
\label{lemmaK1}Every element in $K_{2}(F)$ can be written as $\beta
=\{\pi,u_{0}\}+\sum_{i=1}^{r}\{u_{i},v_{i}\}$ for some $r$ and $u_{i},v_{i}%
\in\mathcal{O}_{F}^{\times}$.
\end{lemma}

\begin{proof}
Standard. By linearity it suffices to prove the claim for a pure symbol
$\{x,y\}$ with $x,y\in F^{\times}$. We may write $x=\pi^{a}u$ and $y=\pi^{b}v$
for $a,b\in\mathbb{Z}$ and $u,v\in\mathcal{O}_{F}^{\times}$, so by linearity%
\begin{align*}
\{x,y\}  &  =\{\pi^{a},\pi^{b}\}+\{\pi^{a},v\}+\{u,\pi^{b}\}+\{u,v\}\\
&  =ab\{\pi,\pi\}+\{\pi,v^{a}\}+\{\pi,u^{-b}\}+\{u,v\}\text{.}%
\end{align*}
The Steinberg relation implies that $\{\pi,\pi\}=\{\pi,-1\}$, so this
simplifies to $\{\pi,(-1)^{ab}v^{a}u^{-b}\}+\{u,v\}$, proving the claim.
\end{proof}

The following is well-known to experts, and in much broader generality. We
record it in the version suitable for us.

\begin{lemma}
\label{lemmaK2}Let $\pi$ be a uniformizer of $F$. Then for any $u\in U_{F}%
^{2}$ there exist $a,b\in U_{F}^{1}$ such that $\{\pi,u\}=\{a,b\}$ holds in
$K_{2}(F)$.
\end{lemma}

\begin{proof}
Since $u\in U_{F}^{2}$, we may write $u=1-z$ with $z\in\mathcal{O}_{F}$ of
valuation $v(z)\geq2$. Define $g:=1+z\pi^{-1}-z$. Note that since $v(z)\geq2$,
the term $z\pi^{-1}$, and thus $z\pi^{-1}-z$, lies in $\mathfrak{m}$. Hence,
$g\in U_{F}^{1}$. We compute%
\[
(1-\pi)\cdot(1-z)=1-\pi-z+\pi z=1-\pi(1+z\pi^{-1}-z)=1-\pi g\text{,}%
\]
i.e. $1-z=\frac{1-\pi g}{1-\pi}$. Now we compute%
\[
\{\pi,u\}=\{\pi,1-z\}=\left\{  \pi,\frac{1-\pi g}{1-\pi}\right\}  =\{\pi,1-\pi
g\}
\]
by the Steinberg relation $\{\pi,1-\pi\}=0$. Using the Steinberg relation a
second time, we obtain%
\[
=\{\pi g,1-\pi g\}-\{g,1-\pi g\}=-\{g,1-\pi g\}\text{,}%
\]
but $g\in U_{F}^{1}$ and $1-\pi g\in U_{F}^{1}$, so take $a:=g^{-1}$ and
$b:=1-\pi g$.
\end{proof}

We need the following crucial observation of Hasse on the higher unit
filtration and $p$-power maps.\ We let%
\[
e_{1}:=\frac{e}{p-1}\text{,}%
\]
where $e:=e_{F/\mathbb{Q}_{p}}$ is the absolute ramification index.

\begin{lemma}
[Hasse]\label{lemma_Hasse}The filtration of the higher unit groups is
compatible with the $p$-power filtration in the following sense:

\begin{enumerate}
\item Suppose $i\geq1$. Then for all sufficiently large $k$ we have%
\[
(U_{F}^{1})^{p^{k}}\subseteq U_{F}^{i}\text{.}%
\]

\item Suppose $i\geq1$. Then whenever $k\geq1$ and $i>pe_{1}+(k-1)e$ one has
the inclusion%
\[
U_{F}^{i}\subseteq(U_{F}^{1})^{p^{k}}\text{.}%
\]

\end{enumerate}
\end{lemma}

\begin{proof}
The crucial ingredient is a case-by-case study of the map $x\mapsto x^{p}$
which one can find explained in \cite[Chapter I, (5.7) Proposition]%
{MR1915966}. The cases are as follows: If $1\leq t\leq e_{1}$, then
$(U_{F}^{t})^{p}\subseteq U_{F}^{pt}$ and if $t>e_{1}$, then $(U_{F}^{t}%
)^{p}\subseteq U_{F}^{t+e}$. As $p\geq2$ and $e\geq1$, it is clear that after
taking $p$-th powers often enough, we land in the second case, and then can
reach arbitrarily high up in the filtration. This settles (1). We could be
more precise about how big $k$ has to get, but we have no use for this
information. For (2) one uses that once $t>e_{1}$ the $p$-power map%
\[
U_{F}^{t}/U_{F}^{t+1}\longrightarrow U_{F}^{t+e}/U_{F}^{t+e+1}\text{,}\qquad
x\mapsto x^{p}%
\]
is surjective, see \cite[Chapter I, (5.7) Proposition, part (3)]{MR1915966}
for the proof. Thus, $U_{F}^{t+e}=(U_{F}^{t})^{p}\cdot\underline{U_{F}%
^{t+e+1}}$. As the same holds for any greater $t$ as well, we may apply the
same computation to the underlined term, giving%
\[
U_{F}^{t+e}=(U_{F}^{t})^{p}\cdot\underline{(U_{F}^{t+1})^{p}U_{F}^{t+e+2}}%
\]
and since $(U_{F}^{t+1})^{p}\subseteq(U_{F}^{t})^{p}$, we get $U_{F}%
^{t+e}=(U_{F}^{t})^{p}U^{t+e+l}$ for any $l\geq1$. Letting $l\rightarrow
+\infty$ and using $\bigcap_{c\geq t}U_{F}^{c}=\{1\}$, we obtain the crucial
equality%
\[
U_{F}^{t+e}=(U_{F}^{t})^{p}\qquad\text{(once }t>e_{1}\text{).}%
\]
Hence, if $i-ke>e_{1}$, we get (by inductively running through the above
identity from right to left for varying $t$)%
\[
(U_{F}^{1})^{p^{k}}\supseteq(U_{F}^{i-ke})^{p^{k}}=(U_{F}^{i-ke+e})^{p^{k-1}%
}=\cdots=U_{F}^{i}\text{.}%
\]
However, $i-ke>e_{1}$ is equivalent to $i>pe_{1}+(k-1)e$. This settles (2).
\end{proof}

We write%
\[
h\colon K_{2}(F)\longrightarrow\mu(F)\text{.}%
\]
for the Hilbert symbol. We had recalled its definition in
\S \ref{sect_RecallHilbertSymbol}.

\begin{theorem}
[Moore]\label{thmMoore}The map $h$ is surjective and $\ker(h)$ is a divisible group.
\end{theorem}

The proof is not terribly difficult. We refer to \cite[Chapter IX, (4.3)
Theorem]{MR1915966} for a textbook proof. In fact, it is also known that the
kernel is uniquely divisible, but this is much harder to show and we shall not
need it.

Next, using our toolbox from Algebraic $K$-theory as in
\S \ref{sect_appendix_Kthy} there is a localization sequence of spectra%
\begin{equation}
K_{\mathfrak{m}}(\mathcal{O}_{F})\longrightarrow K(\mathcal{O}_{F}%
)\longrightarrow K(F)\text{.} \label{lcim2a}%
\end{equation}
It induces a long exact sequence in $K$-theory groups and the boundary map%
\[
K_{2}(F)\overset{\partial}{\longrightarrow}K_{\mathfrak{m},1}(\mathcal{O}%
_{F})\overset{(\diamond)}{\cong}K_{1}(\kappa)\cong\kappa^{\times}%
\]
for $\kappa:=\mathcal{O}_{F}/\mathfrak{m}$ the residue field, is known as the
\emph{tame symbol}. An explicit formula for this map is known, see Equation
\ref{lee1}. The isomorphism $(\diamond)$ stems from devissage and is available
because $\mathcal{O}_{F}$ is regular. It has no counterpart for the rings
$R_{m}$ in general.

\begin{lemma}
\label{lemma_HilbertBecomesTameSymbol}For $n:=\#\mu(F)$ write $n=p^{k}n_{0}$
with $(n_{0},p)=1$. Then the diagram%
\begin{equation}%
\xymatrix{
K_2(F) \ar@{->>}[r]^-{h} \ar[dr]_-{\partial} & \mu(F) \ar@{->>}[d]^-{p^k} \\
& \kappa^{\times}
}
\label{lcx1}%
\end{equation}
is commutative.\ The downward arrow refers to taking the $p^{k}$-th power.
\end{lemma}

\begin{proof}
For any $m\mid n$ write $(x,y)_{m}:=\frac{\operatorname*{Art}_{F}%
(x)\sqrt[m]{y}}{\sqrt[m]{y}}$ for the $m$-Hilbert symbol. This is the variant
of the Hilbert symbol taking a root of possibly lesser order than would be
maximally permitted to obtain a Kummer extension. Then $n=p^{k}(q-1)$ with
$q:=\#\kappa$ and then $h(x,y)^{p^{k}}=(x,y)_{n}^{p^{k}}=(x,y)_{q-1}$ (by
\cite[Chapter IV, (5.1)\ Prop., (7)]{MR1915966}) and $(x,y)_{q-1}%
=\partial\{x,y\}$ (by \cite[Chapter IV, (5.3)\ Theorem, for $n$ taken to be
$q-1$]{MR1915966}).
\end{proof}

\begin{corollary}
\label{cor_HilbIntegralPart}The map $K_{2}(\mathcal{O}_{F})\rightarrow
K_{2}(F)$ is injective, so we may regard $K_{2}(\mathcal{O}_{F})$ as a
subgroup of $K_{2}(F)$. The restriction $h\mid_{K_{2}(\mathcal{O}_{F})}$ is a
surjective map%
\[
h\mid_{K_{2}(\mathcal{O}_{F})}\colon K_{2}(\mathcal{O}_{F})\longrightarrow
\mu(F)[p^{\infty}]
\]
and the kernel is a divisible group.
\end{corollary}

\begin{proof}
Using the localization sequence of Equation \ref{lcim2a}, we obtain%
\[%
\xymatrix{
\cdots\ar[r]^-{j} & K_2(\mathcal{O}_F) \ar[r]^-{i} \ar[dr] & K_2(F) \ar
[r]^-{\partial} \ar[d]^{h} & K_1(\kappa) \ar@{-->}[dl] \ar[r] & 0 \\
& & \mu(F),
}%
\]
where the dashed arrow is the Teichm\"{u}ller lift, i.e. the fact that the
multiplicative group of the residue field, $\mathbb{F}_{q}^{\times}$, can be
identified isomorphically with the prime-to-$p$ roots of unity $\left\langle
\zeta_{q-1}\right\rangle $ in\ Equation \ref{lc1a}. This is a section to the
downward arrow on the right in Diagram \ref{lcx1}. We see that the image of
the dashed arrow amounts precisely to the prime-to-$p$ torsion summand of
$\mu(F)$. Since $h$ is surjective, the image of $h\mid_{K_{2}(\mathcal{O}%
_{F})}$ must surject onto the $p$-primary summand. Moreover, we have $j=0$, so
the map $i$ is injective by \cite[Chapter V, Corollary 6.6.2]{MR3076731}. Now
suppose $\alpha\in K_{2}(\mathcal{O}_{F})$ lies in the kernel of $h$. To check
that $\alpha$ is a divisible element, say $\alpha=M\beta$ for a given $M\geq
1$, write $\alpha=M(q-1)\beta^{\prime}$ for some $\beta^{\prime}$ in
$K_{2}(F)$. This is possible since by Theorem \ref{thmMoore} kernel elements
are divisible in $K_{2}(F)$. As the codomain of $\partial$ has order $q-1$, we
must have $\partial\left(  (q-1)\beta^{\prime}\right)  =1$, so $(q-1)\beta
^{\prime}\in K_{2}(\mathcal{O}_{F})$. Thus, $\alpha$ can be written as an
$M$-th multiple of an element in $K_{2}(\mathcal{O}_{F})$.
\end{proof}

\begin{corollary}
\label{cor_K2O0DivisibleForOddPrimes}Suppose $p$ is odd. Then the group
$K_{2}(\mathcal{O}_{0})$ is divisible.
\end{corollary}

\begin{proof}
Apply Corollary \ref{cor_HilbIntegralPart} to the field $F:=F_{0}$. Suppose
$p$ is odd. As $F_{0}/\mathbb{Q}_{p}$ is unramified by construction, it does
not contain any non-trivial $p$-power roots of unity (This is classical.
Observe that otherwise $\mathbb{Q}_{p}(\zeta_{p}-1)/\mathbb{Q}_{p}$ is a
subextension, but its minimal polynomial $((T+1)^{p}-1)/T$ is Eisenstein and
thus this is a non-trivial degree $p-1>1$ totally ramified field extension),
so $\mu(F_{0})[p^{\infty}]=0$. The case $p=2$ is different and the claim would
be false.
\end{proof}

We are ready to prove a key property for our study of the $K$-theory of the
rings $R_{m}$. For $n:=\#\mu(F)$ write $n=p^{k}(q-1)$ with $k\geq0$. Below, we
shall refer to the value of $k$ on several occasions.

\begin{lemma}
\label{lemma_TypeCSymbols}Suppose $p$ is odd. Then for all sufficiently large
$m\geq1$ it holds that for any $u,v\in1+\mathfrak{\tilde{m}}_{m}$ the Hilbert
symbol $h(u,v)$ vanishes. If $k\geq1$, then any%
\[
m>pe_{1}+(k-1)e
\]
is sufficiently large. For $k=0$, any $m\geq1$ works.
\end{lemma}

\begin{proof}
Suppose $k=0$. Then $F$ has no non-trivial $p$-power roots of unity. Hence, by
Corollary \ref{cor_HilbIntegralPart} the group $K_{2}(\mathcal{O}_{F})$ is
divisible. By the factorization%
\begin{equation}
K_{2}(R_{m})\longrightarrow K_{2}(\mathcal{O}_{F})\overset{h}{\longrightarrow
}\mu(F) \label{lcd0}%
\end{equation}
coming from $R_{m}\subseteq\mathcal{O}_{F}$, the Hilbert symbol must vanish
(the image of a divisible element is divisible, but since $\mu(F)$ is a finite
group, it follows that the map $h\mid_{K_{2}(\mathcal{O}_{F})}$ is zero). This
case was easy. For the rest of the proof we suppose $k\geq1$. By Lemma
\ref{lemma_Hasse} we have%
\begin{equation}
U_{F}^{m}\subseteq(U_{F}^{1})^{p^{k}}\text{.} \label{lcd}%
\end{equation}
Now write%
\[
u=1-px+\pi^{m}y\qquad\text{(}x\in\mathcal{O}_{0}\text{ and }y\in
\mathcal{O}_{F}\text{)}%
\]
using that $\mathfrak{\tilde{m}}_{m}=p\mathcal{O}_{0}+\mathfrak{m}%
^{m}\mathcal{O}_{F}$ by Equation \ref{lc1c}. Then%
\begin{align*}
u  &  =\left(  1-px\right)  \cdot\left(  \frac{1-px+\pi^{m}y}{1-px}\right)
=(1-px)\cdot\left(  1+\frac{\pi^{m}y}{1-px}\right) \\
&  =(1-px)\cdot\left(  1+\pi^{m}y\sum_{l\geq0}(px)^{l}\right)  =(1-px)\cdot
\beta\text{.}%
\end{align*}
Note that $\beta\in1+\mathfrak{m}^{m}\mathcal{O}_{F}$, i.e. $\beta\in
U_{F}^{m}$. We can perform an analogous decomposition for $v$, say
$v=(1-py)\cdot\beta^{\prime}$ with $\beta^{\prime}\in U_{F}^{m}$. Then%
\[
\{u,v\}=\{1-px,1-py\}+\{1-px,\beta^{\prime}\}+\{\beta,1-py\}+\{\beta
,\beta^{\prime}\}\text{.}%
\]
We now work in $K_{2}(F)$. As $\beta\in U_{F}^{m}$, we may write $\beta
=\beta_{0}^{p^{k}}$ by Equation \ref{lcd}, and as $U_{F}^{1}$ is a
$\mathbb{Z}_{p}$-module, even $\beta=\beta_{00}^{p^{k}(q-1)}=\beta_{00}^{n}$.
The same works for $\beta^{\prime}$. Thus, the term $\{1-px,\beta^{\prime
}\}+\{\beta,1-py\}+\{\beta,\beta^{\prime}\}$ is an $n$-th multiple of
something. Therefore under $h$, which is a map to a group of order $n$, these
elements must map to zero. It remains to check that%
\[
h(1-px,1-py)=0\text{,}%
\]
but this is clear since $1-px,1-py\in\mathcal{O}_{0}$, so our element comes
from the left under the composition $K_{2}(\mathcal{O}_{0})\rightarrow
K_{2}(R_{m})\rightarrow K_{2}(F)$, but $K_{2}(\mathcal{O}_{0})$ is divisible
by Corollary \ref{cor_K2O0DivisibleForOddPrimes}, and mapping a divisible
element to a group of order $n$ again must map the element to zero. This is
the same idea as we had used in Equation \ref{lcd0}, but with a reversed
direction of inclusion, $\mathcal{O}_{0}\subseteq R_{m}$. This finishes the proof.
\end{proof}

\begin{proposition}
\label{prop_1}Suppose $p$ is odd. For $m\geq1$ sufficiently big, the Hilbert
symbol vanishes on the image of $K_{2}(R_{m})\rightarrow K_{2}(F)$. If
$k\geq1$, then any%
\[
m>pe_{1}+(k-1)e
\]
is sufficiently large. For $k=0$, any $m\geq1$ works.
\end{proposition}

\begin{proof}
We had shown that $R_{m}$ is a local ring in Lemma \ref{lemma_Rm_main}. By
Equation \ref{lc1aa} we have%
\begin{equation}
R_{m}^{\times}\cong\left\langle \zeta_{q-1}\right\rangle \times
(1+\mathfrak{\tilde{m}}_{m})\text{.} \label{lmix1}%
\end{equation}
Now by Proposition \ref{prop_DennisStein} the Hilbert symbol vanishes on
$K_{2}(R_{m})$ once it vanishes on all pure symbols $\{a,b\}$ with $a,b\in
R_{m}^{\times}$. Using the decomposition in Equation \ref{lmix1} it suffices
to check this for%
\[
\text{(A) }\{\zeta_{q-1},\zeta_{q-1}\}\text{,}\qquad\text{(B) }\{\zeta
_{q-1},u\}\text{,}\qquad\text{(C) }\{u,v\}
\]
with $u,v\in1+\mathfrak{\tilde{m}}_{m}$. The symbols of type (A) and (B) are
zero by adapting \cite[Chapter IX, (4.1) Lemma]{MR1915966} to $R_{m}$. We
provide the details: (B) All pure symbols $\{\zeta_{q-1},u\}$ with
$u\in1+\mathfrak{\tilde{m}}_{m}$ are zero. To see this, note that $\zeta
_{q-1}$ is $(q-1)$-torsion, but since $1+\mathfrak{\tilde{m}}_{m}$ is a
$\mathbb{Z}_{p}$-module\ (Lemma \ref{lemma_PrincipalUnitsZpModuleStructure})
and $\frac{1}{q-1}\in\mathbb{Z}_{p}$, we can take a root,%
\[
\{\zeta_{q-1},u\}=\{\zeta_{q-1},(u^{\frac{1}{q-1}})^{q-1}\}=\{\zeta
_{q-1}^{q-1},u^{\frac{1}{q-1}}\}=0\text{.}%
\]
Moreover, the symbols (A) are zero. To this end, copy the proof that
$K_{2}(\mathbb{F}_{q})=0$, which holds up to an error in $1+p\mathcal{O}_{0}$
in $\mathcal{O}_{0}$, but $1+p\mathcal{O}_{0}$ is also a $\mathbb{Z}_{p}%
$-module under exponentiation, so the same argument as for (B) works. Thus, it
remains to consider the symbols of type (C). This case is covered by\ Lemma
\ref{lemma_TypeCSymbols}.
\end{proof}

\begin{lemma}
\label{Lemma_DivisibleElementsLieInAllK2}Suppose $\alpha\in K_{2}(F)$ is a
divisible element in the group. Then for any $c\geq0$ the element lies in the
image of $K_{2}(R_{c})\rightarrow K_{2}(F)$ induced from the inclusion
$R_{c}\subset F$.
\end{lemma}

\begin{proof}
Fix any integer $c\geq2$. Pick $M:=(q-1)M_{0}$ such that $M_{0}$ is a
sufficiently large $p$-power such that%
\[
\left(  U_{F}^{1}\right)  ^{M_{0}}\subseteq U_{F}^{c}%
\]
holds. This is possible by Lemma \ref{lemma_Hasse}. As Equation \ref{lc1a}
asserts that $\mathcal{O}_{F}^{\times}\cong\left\langle \zeta_{q-1}%
\right\rangle \times U_{F}^{1}$, we deduce that $\left(  \mathcal{O}%
_{F}^{\times}\right)  ^{M}\subseteq U_{F}^{c}$. As $\alpha$ is divisible, we
find some $\beta\in K_{2}(F)$ such that%
\[
\alpha=M^{3}\cdot\beta
\]
holds in $K_{2}(F)$. We may write $\beta=\{\pi,u_{0}\}+\sum_{i=1}^{r}%
\{u_{i},v_{i}\}$ for some $u_{i},v_{i}\in\mathcal{O}_{F}^{\times}$ by Lemma
\ref{lemmaK1}. Now,%
\[
M^{2}\cdot\{u_{i},v_{i}\}=\{u_{i}^{M},v_{i}^{M}\}\in\{U_{F}^{c},U_{F}^{c}\}
\]
and%
\[
M^{3}\cdot\{\pi,u_{0}\}=M^{2}\cdot\{\pi,u_{0}^{M}\}=M^{2}\cdot\{a,b\}=\{a^{M}%
,b^{M}\}\in\{U_{F}^{c},U_{F}^{c}\}
\]
with $a,b\in U_{F}^{1}$ supplied by\ Lemma \ref{lemmaK2} (using that
$u_{0}^{M}\in U_{F}^{2}$). Since $U_{F}^{c}=1+\mathfrak{m}^{c}\mathcal{O}%
_{F}\subseteq R_{c}$, we obtain that $\alpha=M^{3}\beta\in\operatorname*{im}%
K_{2}(R_{c})$. This proves our claim for $c\geq2$. The smaller the $c$, the
bigger the ring $R_{c}$, so we get the claim for all $c\geq0$.
\end{proof}

Since $(R_{m},\mathfrak{\tilde{m}}_{m})$ is a local ring with field of
fractions $F$ by Lemma \ref{lemma_Rm_main}, we get a localization fibration
sequence%
\begin{equation}
K_{\mathfrak{\tilde{m}}_{m}}(R_{m})\longrightarrow K(R_{m})\longrightarrow
K(F) \label{lgitto4}%
\end{equation}
of spectra. It induces a long exact sequence of homotopy groups%
\[
\cdots\longrightarrow K_{2}(R_{m})\longrightarrow K_{2}(F)\overset{\partial
}{\longrightarrow}K_{\mathfrak{\tilde{m}}_{m},1}(R_{m})\overset{\beta
}{\longrightarrow}K_{1}(R_{m})\overset{\gamma}{\longrightarrow}K_{1}%
(F)\longrightarrow\cdots\text{,}%
\]
where $\partial$ denotes the connecting homomorphism $\partial\colon
K(F)\rightarrow\Sigma K_{\mathfrak{\tilde{m}}_{m}}(R_{m})$. The greek letters
serve the purpose to refer to these arrows below.

\begin{theorem}
\label{thm_LocalMain}Suppose $p$ is odd.

\begin{enumerate}
\item Then the sequence of subgroups%
\[
\operatorname*{im}K_{2}(R_{0})\supseteq\operatorname*{im}K_{2}(R_{1}%
)\supseteq\operatorname*{im}K_{2}(R_{2})\supseteq\cdots
\]
of $K_{2}(F)$ becomes stationary after finitely many steps. In particular,
there is a unique minimal $m_{0}\geq0$ such that%
\begin{equation}
\operatorname*{im}K_{2}(R_{m_{0}})=\bigcap_{c\geq0}\operatorname*{im}%
K_{2}(R_{c})\text{.} \label{lcd1}%
\end{equation}

\item The group in Equation \ref{lcd1} agrees with the kernel of the Hilbert
symbol on $K_{2}(F)$.

\item If $F/\mathbb{Q}_{p}$ is unramified, $m_{0}=0$.

\item For any $m\geq m_{0}$ there is a canonical isomorphism
\[
\phi\colon K_{\mathfrak{\tilde{m}}_{m},1}(R_{m})\cong\mu(F)
\]
such that%
\begin{equation}%
\xymatrix{
\cdots\ar[r] & K_2(R_m) \ar[r] \ar[dr]_{0} & K_2(F) \ar[r]^-{\partial}
\ar[d]_{h}
& K_{\tilde{\mathfrak{m}}_m,1}(R_m) \ar[r] \ar@{-->}[dl]^{\phi} & 0 \\
& & \mu(F)
}
\label{lcd1a}%
\end{equation}
commutes.
\end{enumerate}
\end{theorem}

\begin{proof}
(Claim 1) By Proposition \ref{prop_1} for $m$ sufficiently big we obtain that
$h(\operatorname*{im}K_{2}(R_{m}))=0$. By Theorem \ref{thmMoore} it follows
that for any such $m$ the elements in $\operatorname*{im}\left(  K_{2}%
(R_{m})\right)  $ are divisible elements in $K_{2}(F)$. By Lemma
\ref{Lemma_DivisibleElementsLieInAllK2} any divisible element lies in
$\bigcap_{c\geq0}\operatorname*{im}K_{2}(R_{c})$. Thus, after finitely many
steps the sequence becomes stationary and then equals $\bigcap_{c\geq
0}\operatorname*{im}K_{2}(R_{c})$. Pick $m_{0}$ to be the minimal
such.\newline(Claim 2) Since any element in $\operatorname*{im}K_{2}(R_{c})$
for $c$ sufficiently large must be divisible by Proposition \ref{prop_1}, but
the Hilbert symbol maps to a finite group, the group in Equation \ref{lcd1} is
contained in the kernel of $h$. Again by Theorem \ref{thmMoore} the kernel of
the Hilbert symbol consists of divisible elements, so Lemma
\ref{Lemma_DivisibleElementsLieInAllK2} yields the reverse inclusion.\newline%
(Claim 3) If $F/\mathbb{Q}_{p}$ is unramified, then $F=F_{0}$, so
$R_{0}=\mathcal{O}_{F}=\mathcal{O}_{0}$ and by Lemma
\ref{cor_K2O0DivisibleForOddPrimes} $K_{2}(R_{0})=K_{2}(\mathcal{O}_{0})$ is a
divisible group, so already on the largest possible group in our filtration
the Hilbert symbol must be trivial, i.e. $m_{0}=0$ already does the trick
according to Claim 2.\newline(Claim 4) Since $R_{m}$ is a local ring and $F$ a
field, we may use \cite[Example 1.6]{MR2371852} to see that the determinant
induces the downward isomorphisms in%
\[%
\xymatrix{
K_1(R) \ar[d]_{\cong} \ar[r]^{\gamma} & K_1(F) \ar[d]^{\cong} \\
R^{\times} \ar@{^{(}->}[r] & F^{\times},
}%
\]
showing that $\gamma$ is injective, and thus $\beta=0$. Hence, we obtain
Figure \ref{lcd1a}. The top row now yields%
\[
K_{2}(F)/\operatorname*{im}K_{2}(R_{m})\underset{\sim}{\overset{\partial
}{\longrightarrow}}K_{\mathfrak{\tilde{m}}_{m},1}(R_{m})\text{.}%
\]
However, by Claim 1 and Claim 2, the group $\operatorname*{im}K_{2}(R_{m})$ is
precisely $\ker(h)$, and the Hilbert symbol, the vertical arrow $h$, is
surjective by Theorem \ref{thmMoore}, so the map $\phi$ exists by the
universal property of cokernels and must be an isomorphism as well.
\end{proof}

\begin{definition}
[Local optimal order]\label{def_LocalOptimalOrder}If $p$ is odd, define
$\mathcal{R}:=R_{m_{0}}$ with $m_{0}$ as in Theorem \ref{thm_LocalMain}. For
$p=2$, take $\mathcal{R}:=\mathcal{O}_{F}$. We call $\mathcal{R}$ the
\emph{optimal order} in $F$.
\end{definition}

\section{Global theory}

We now transport the local theory to the global setting using ad\`{e}les.

\subsection{A family of global orders}

Let $F$ be a number field. If $v$ is a place, we write $F_{v}$ for the
completion at $v$. The ad\`{e}les are the restricted product%
\[
\mathbf{A}_{F}:=\left.  \underset{v\;}{\prod\nolimits^{\prime}}\right.  F_{v}%
\]
over all places. If $v$ is a finite place, let $\mathcal{O}_{v}$ denote the
ring of integers of $F_{v}$, otherwise $\mathcal{O}_{v}:=F_{v}$.

There is an exact sequence%
\begin{equation}
0\longrightarrow\mathcal{O}_{F}\overset{\operatorname*{diag}}{\longrightarrow
}F\oplus\prod_{v}\mathcal{O}_{v}\overset{\operatorname*{diff}}{\longrightarrow
}\mathbf{A}_{F}\longrightarrow0 \label{lglobaladelesequence}%
\end{equation}
with $\operatorname*{diag}(x):=(x,x,\ldots)$ the diagonal map.

Suppose $\underline{m}=(m_{v})_{v}$ is an effective Weil divisor on
$\operatorname*{Spec}\mathcal{O}_{F}$. We can define%
\begin{equation}
\mathcal{R}_{\underline{m}}:=F\cap\left(
{\textstyle\prod_{v}}
R_{v,m_{v}}\right)  \text{,} \label{lcqx2}%
\end{equation}
where the meaning of \textquotedblleft$R_{v,m_{v}}$\textquotedblright\ is as
follows: for a finite place we mean $R_{m_{v}}\subseteq\mathcal{O}_{v}$ in the
sense of Equation \ref{lcqx1}, and for an infinite place we mean $F_{v}$. For
the divisor $\underline{m}=0$ we get $\mathcal{R}_{\underline{m}}%
=\mathcal{O}_{F}$ by Equation \ref{lglobaladelesequence}. If $\underline
{m}^{\prime}\geq\underline{m}$ are effective Weil divisors on
$\operatorname*{Spec}\mathcal{O}_{F}$, we obtain a commutative diagram%
\begin{equation}%
\xymatrix{
0 \ar[r] & \mathcal{R}_{\underline{m}^{\prime}} \ar@{>->}[d] \ar
[r] & F \oplus\prod_{v} R_{v,m_v^{\prime}} \ar@{>->}[d] \ar[r] & \mathbf
{A}_F \ar[r] \ar@{=}[d] & 0 \\
0 \ar[r] & \mathcal{R}_{\underline{m}} \ar[r] & F \oplus\prod_{v} R_{v,m_v}
\ar[r] & \mathbf{A}_F \ar[r] & 0. \\
}
\label{lvvd3}%
\end{equation}
The exactness of the rows is clear, except perhaps the surjectivity on the
right. Suppose $(\alpha_{v})_{v}\in\mathbf{A}_{F}$ is an ad\`{e}le. Then for
any integer $N\geq1$ the vector $(\frac{1}{N}\alpha_{v})_{v}$ is also an
ad\`{e}le, because the ad\`{e}le ring $\mathbf{A}_{F}$ is a $\mathbb{Q}%
$-algebra. Since the right arrow in Equation \ref{lglobaladelesequence} is
surjective, we find $\beta\in F$ and $(\beta_{v})_{v}\in\prod\mathcal{O}_{v}$
such that $(\frac{1}{N}\alpha_{v})_{v}=(\beta-\beta_{v})_{v}$. In particular,
$(N\beta-N\beta_{v})_{v}$ maps to $(\alpha_{v})_{v}$, but $N\beta\in F$ and
once $N$ is sufficiently divisible by prime factors lying below those places
with $m_{v}\gneqq0$, we get $N\beta_{v}\in R_{v,m_{v}}$ for all $v$.

\begin{lemma}
The set $\mathcal{R}_{\underline{m}}\subseteq\mathcal{O}_{F}$ is a
one-dimensional Noetherian domain of finite index $[\mathcal{O}_{F}%
:\mathcal{R}_{\underline{m}}]<\infty$ and with field of fractions $F$. It is a
finitely generated free $\mathbb{Z}$-module of rank $[F:\mathbb{Q}]$.
\end{lemma}

\begin{proof}
Harmless. Consider the quotient of the bottom row by the top row in Diagram
\ref{lvvd3} for the choice $\underline{m}=0$ and $\underline{m}^{\prime}$
equal to the $\underline{m}$ in the claim of this lemma. The resulting
downward injections are an isomorphism on $F$ and $\mathbf{A}_{F}$, so
$[\mathcal{O}_{F}:\mathcal{R}_{\underline{m}}]=\prod_{v}[\mathcal{O}%
_{v}:R_{v,m_{v}}]<\infty$, showing that $\mathcal{O}_{F}$ is a finite
$\mathcal{R}_{\underline{m}}$-algebra. It follows that $\mathcal{R}%
_{\underline{m}}$ is a Noetherian one-dimensional ring. For any $x\in F$, we
can write $x=\frac{a}{b}$ with $a,b\in\mathcal{O}_{F}$ and then $x=\frac
{ga}{gb}$ with $ga,gb\in\mathcal{R}_{\underline{m}}$, so $\operatorname*{Frac}%
\mathcal{R}_{\underline{m}}=F$.
\end{proof}

Now fix $\underline{m}$. Let $\underline{\gamma}$ be a further effective Weil
divisor on $\operatorname*{Spec}\mathcal{O}_{F}$. Define%
\[
I_{\underline{\gamma}}:=F\cap\left(
{\textstyle\prod_{v}}
\mathfrak{\tilde{m}}_{v}^{\underline{\gamma}_{v}}\right)  \text{,}%
\]
where at the finite places $\mathfrak{\tilde{m}}_{v}$ refers to the unique
maximal ideal $\mathfrak{\tilde{m}}_{m_{v}}$ of the local ring $R_{v,m_{v}}$,
and we let $\mathfrak{\tilde{m}}_{v}:=F_{v}$ at the infinite places. Since
$\underline{\gamma}$ is a Weil divisor, we have $\mathfrak{\tilde{m}}%
_{v}^{\underline{\gamma}_{v}}=(1)$ for all but finitely many places. It is
clear that $I_{\underline{\gamma}}\subseteq\mathcal{R}_{\underline{m}}$ is an
ideal. We obtain a diagram, similar to Diagram \ref{lvvd3},%
\begin{equation}%
\xymatrix{
0 \ar[r] & I_{\underline{\gamma}} \ar@{>->}[d] \ar[r] & F \oplus\prod_{v}
{\tilde{\mathfrak{m}}}_{v}^{\underline{\gamma}_{v}} \ar@{>->}[d] \ar
[r] & \mathbf{A}_F \ar[r] \ar@{=}[d] & 0 \\
0 \ar[r] & \mathcal{R}_{\underline{m}} \ar[r] \ar@{->>}[d] & F \oplus\prod_{v}
R_{v,m_v} \ar[r]  \ar@{->>}[d] & \mathbf{A}_F \ar[r] & 0 \\
& \mathcal{R}_{\underline{m}}/I_{\underline{\gamma}} \ar[r]^-{\sim}
& \prod_{v} R_{v,m_v}/{\tilde{\mathfrak{m}}}_{v}^{\underline{\gamma}_{v}},
}
\label{lcgx3}%
\end{equation}
where the downward arrows between the first rows are just the inclusions. The
exactness of the middle row holds for the same reason as before, since for any
$c\geq1$ we can also pick $N$ sufficiently divisible such that $N\beta_{v}%
\in\mathfrak{\tilde{m}}_{v}^{c}$. The downward arrows induce isomorphisms on
the summand $F$ as well as on $\mathbf{A}_{F}$. It follows that the bottom
horizontal arrow is indeed an isomorphism. Since $\mathfrak{\tilde{m}}%
_{v}^{\underline{\gamma}_{v}}=(1)$ for all but finitely many places, the
product in the bottom row is over a finite set, and moreover we deduce that
$\mathcal{R}_{\underline{m}}/I_{\underline{\gamma}}$ is a finite ring.

\begin{lemma}
\label{lemma_topcompare}We keep $\underline{m}$ fixed. The inclusion of
$\mathcal{R}_{\underline{m}}$ into the product coming from Equation
\ref{lcqx2} induces an isomorphism of topological rings%
\begin{equation}
\underset{\underline{\gamma}}{\underleftarrow{\lim}}\mathcal{R}_{\underline
{m}}/I_{\underline{\gamma}}\cong\prod_{v}R_{v,m_{v}}\text{,} \label{lcggx2}%
\end{equation}
where the finite rings $\mathcal{R}_{\underline{m}}/I_{\underline{\gamma}}$
are equipped with the discrete topology, $\underline{\gamma}$ runs through all
effective Weil divisors on $\operatorname*{Spec}\mathcal{O}_{F}$ (partially
ordered by $\leq$), the limit is given the inverse limit topology,
$R_{v,m_{v}}$ is taken with its natural subspace topology from Lemma
\ref{lemma_RmIsClopenInF}, and $\prod$ is given the product topology.
Moreover, on the left side the inverse limit agrees with the profinite
completion of $\mathcal{R}_{\underline{m}}$, seen as an abelian group.
\end{lemma}

\begin{proof}
The bottom isomorphism in Diagram \ref{lcgx3} yields isomorphisms%
\begin{equation}
\mathcal{R}_{\underline{m}}/I_{\underline{\gamma}}\cong\prod_{v}R_{v,m_{v}%
}/\mathfrak{\tilde{m}}_{v}^{\underline{\gamma}_{v}} \label{lcggx1}%
\end{equation}
and both sides carry the discrete topology (also on the right, using that
$\mathfrak{\tilde{m}}_{v}^{\underline{\gamma}_{v}}$ is clopen in $R_{v,m_{v}}%
$). Thus, running over the inverse system of effective Weil divisors, we have
levelwise isomorphisms, and thus obtain an isomorphism in the limit. As
$\prod_{v}$ is a limit itself, limits commute with each other and
$\underset{c}{\underleftarrow{\lim}}R_{v,m_{v}}/\mathfrak{\tilde{m}}_{v}%
^{c}=R_{v,m_{v}}$, the first claim follows. Note that the topology on
$R_{v,m_{v}}$ comes from the valuation on $\mathcal{O}_{v}$, which can be
sandwiched with the filtration by powers of $\mathfrak{\tilde{m}}_{v}^{c}$, so
the inverse limit topology agrees with the valuation topology, and
$R_{v,m_{v}}$ is adically complete (it is clopen in $\mathcal{O}_{v}$ by Lemma
\ref{lemma_RmIsClopenInF} and compact subspaces of complete metric spaces are
themselves complete). Finally, to see that this inverse limit agrees with the
profinite completion, we just need to use that each $\mathcal{R}%
_{\underline{m}}/I_{\underline{\gamma}}$ is a finite quotient, and conversely
for any finite quotient $\mathcal{R}_{\underline{m}}/N$ (as an abelian
group!), we can pick $M$ to be the exponent and then the Weil divisor of $M$
on $\mathbb{Z}$ pulls back to a divisor on $\operatorname*{Spec}%
\mathcal{O}_{F}$ such that $I_{\underline{\gamma}}\subseteq(N)\mathcal{O}_{F}%
$. Thus, the $I_{\underline{\gamma}}$ are a cofinal family.
\end{proof}

The ring $\mathcal{R}_{\underline{m}}$ is one-dimensional Noetherian, so every
non-zero prime ideal is maximal. In particular, any two distinct non-zero
prime ideals are automatically coprime. Thus, the Chinese Remainder Theorem
yields isomorphisms $\mathcal{R}_{\underline{m}}/I_{\underline{\gamma}}%
\cong\prod_{\mathcal{I}}\mathcal{R}_{\underline{m}}/\mathcal{I}$, where
$\mathcal{I}$ runs through the minimal primes of the support of $\mathcal{R}%
_{\underline{m}}/I_{\underline{\gamma}}$, viewed as an $\mathcal{R}%
_{\underline{m}}$-module. In the inverse limit, this yields a ring isomorphism%
\begin{equation}
\underset{\underline{\gamma}}{\underleftarrow{\lim}}\mathcal{R}_{\underline
{m}}/I_{\underline{\gamma}}\cong\prod_{\mathcal{P}}F_{\mathcal{P}}%
\text{,}\qquad\text{where}\qquad F_{\mathcal{P}}:=\underset{i}{\underleftarrow
{\lim}}\mathcal{R}_{\underline{m}}/\mathcal{I}_{\mathcal{P},i}\text{,}
\label{lcggx3}%
\end{equation}
where $\mathcal{P}$ runs through the maximal ideals of $\operatorname*{Spec}%
\mathcal{R}_{\underline{m}}$ and $\mathcal{I}_{\mathcal{P},i}$ is a family of
ideals such that the finite modules $\mathcal{R}_{\underline{m}}%
/\mathcal{I}_{\mathcal{P},i}$ have support equal to the point $\mathcal{P}$ in
$\operatorname*{Spec}\mathcal{R}_{\underline{m}}$. Analogous to Lemma
\ref{lemma_topcompare}, the isomorphism in Equation \ref{lcggx3} is seen to be
a homeomorphism.

The product decompositions in Equations \ref{lcggx2} and \ref{lcggx3} are
indexed over different sets: Maximal ideals of $\operatorname*{Spec}%
\mathcal{O}_{F}$ (namely the finite places $v$) versus the maximal ideals
$\mathcal{P}$ of $\operatorname*{Spec}\mathcal{R}_{\underline{m}}$.

\begin{proposition}
\label{prop_Subdecompose}The product decomposition of Equation \ref{lcggx2} is
a refinement of the decomposition in Equation \ref{lcggx3}: The factors
$R_{v,m_{v}}$ do not admit a further non-trivial decomposition as a product
ring, while the $F_{\mathcal{P}}$ split as a finite direct product of rings,
and the factors correspond (in a fashion which can be made canonical) to
finite places $v$ of $\mathcal{O}_{F}$ lying over $\mathcal{P}$.
\end{proposition}

To prove this, we need to discuss the map%
\[
j_{\underline{m}}\colon\operatorname*{Spec}\mathcal{O}_{F}\longrightarrow
\operatorname*{Spec}\mathcal{R}_{\underline{m}}\text{.}%
\]
This is a finite morphism and by the same argument as in Lemma
\ref{lemma_IsIntegralClosure} it is the normalization map. For all maximal
ideals $\mathfrak{p}$ of $\mathcal{O}_{F}$ coprime to the conductor ideal the
map is a local isomorphism, i.e. $(\mathcal{R}_{\underline{m}})_{\mathfrak{p}%
\cap\mathcal{R}_{\underline{m}}}\overset{\sim}{\longrightarrow}(\mathcal{O}%
_{F})_{\mathfrak{p}}$. For the finitely many maximal ideals not coprime to the
conductor all sorts of phenomena can occur. It might be useful to think of the
geometric analogue of a singular curve and a cusp resp. nodal singularity
exhibiting analogous phenomena. We recall a crucial principle:

\begin{proposition}
[{\cite[Theorem 6.5]{MR0241408}}]\label{prop_NormalizAndCompletion}Suppose
$(R,\mathfrak{m})$ is an excellent reduced Noetherian local ring. Then there
is a canonical bijection between the maximal ideals of the normalization
$R^{\prime}$ and the minimal primes of the completion $\widehat{R}%
_{\mathfrak{m}}$.
\end{proposition}

We will not explain the precise construction of the bijection here. We also recall:

\begin{lemma}
[{\cite[Proposition 4.3.2]{MR2266432}}]\label{lemma_2}Suppose $(R,\mathfrak{m}%
)$ is a complete local Noetherian ring and $S$ a finite $R$-algebra. Then $S$
is semilocal and there is an isomorphism of rings $S\cong\prod S_{\mathfrak{q}%
}$, where $\mathfrak{q}$ runs through the maximal ideals of $S$.
\end{lemma}

\begin{proof}
[Proof of Prop. \ref{prop_Subdecompose}]As $j_{\underline{m}}$ is the
normalization, so is the base change to any local ring $j_{\underline{m}%
}\otimes_{\mathcal{R}_{\underline{m}}}\mathcal{R}_{\underline{m},\mathcal{P}}$
for a maximal ideal $\mathcal{P}$ of $\mathcal{R}_{\underline{m}}$. Then
Proposition \ref{prop_NormalizAndCompletion} applies. The completion side,
i.e. $\widehat{(\mathcal{R}_{\underline{m},\mathcal{P}})}_{\mathcal{P}}$, then
corresponds to a factor $F_{\mathcal{P}}=\underset{i}{\underleftarrow{\lim}%
}\mathcal{R}_{\underline{m}}/\mathcal{I}_{\mathcal{P},i}$ in Equation
\ref{lcggx3}. The proposition therefore identifies the minimal primes of
$F_{\mathcal{P}}$ with the maximal ideals of the normalization. However, the
latter exactly correspond to the points in the fiber of the normalization map
$j_{\underline{m}}$ over the point $\mathcal{P}\in\operatorname*{Spec}%
\mathcal{R}_{\underline{m}}$. As $F_{\mathcal{P}}$ is a finite $\widehat
{(\mathcal{R}_{\underline{m},\mathcal{P}})}_{\mathcal{P}}$-algebra, Lemma
\ref{lemma_2} yields a further decomposition. By lifting of idempotents, the
factors correspond to those of $S/\mathfrak{m}S$ in the cited lemma, in
particular each maximal ideal in Lemma \ref{lemma_2} correponds to a unique
minimal ideal of Proposition \ref{prop_NormalizAndCompletion}. The factors
$R_{v,m_{v}}$ are domains, so they do not permit non-trivial decompositions as
a product of rings.
\end{proof}

\begin{definition}
[Global optimal order]\label{def_GlobalOptimalOrder}Suppose $v$ is a finite
place. Write $\mathcal{R}_{v}\subseteq\mathcal{O}_{v}$ for the optimal order
according to Definiton \ref{def_LocalOptimalOrder}. For infinite places $v$,
define $\mathcal{R}_{v}:=\mathcal{O}_{v}=F_{v}$. We define%
\[
\mathcal{R}:=F\cap\left(
{\textstyle\prod_{v}}
\mathcal{R}_{v}\right)  \text{.}%
\]
We call $\mathcal{R}$ the \emph{optimal order} in $F$.
\end{definition}

Of course $\mathcal{R}$ is an order of the type $\mathcal{R}_{\underline{m}}$
for $\underline{m}$ suitably chosen.

\subsection{$K$-theory computations}

We shall imitate Gillet's proof of Weil reciprocity, but with the twist of
exploiting the special feature that our order, while typically not regular, is
sufficiently small to make the wild part of the Hilbert symbol visible in the
boundary maps.

As before, we refer to Appendix \S \ref{sect_appendix_Kthy} for the necessary
background on Algebraic $K$-theory. Let $I=(N)$ be the principal ideal
generated by some squarefree natural number $N\geq1$. As a special instance of
Equation \ref{lvtx1} we get the localization sequence%
\begin{equation}
K_{I}(\mathcal{R})\longrightarrow K(\mathcal{R})\longrightarrow K(\mathcal{R}%
[\frac{1}{N}])\text{.} \label{lcimde7a}%
\end{equation}
We may factor $I=(p_{1}p_{2}\cdots p_{r})$ into prime numbers, so that each
pair $(p_{i})+(p_{j})=1$ is coprime for $i\neq j$. Thus, relying on Equation
\ref{lvtx2}, we get%
\[
\bigoplus K_{(p_{i})}(\mathcal{R})\longrightarrow K(\mathcal{R}%
)\longrightarrow K(\mathcal{R}[\frac{1}{N}])\text{.}%
\]
These sequences sit in a family for $N\mid N^{\prime}$. As $K$-theory commutes
with filtering colimits, we obtain the fiber sequence%
\[
\bigoplus_{p}K_{(p)}(\mathcal{R})\longrightarrow K(\mathcal{R})\longrightarrow
K(F)\text{,}%
\]
where the sum runs over all prime numbers and we have used that
$F=\operatorname*{colim}_{N}\mathcal{R}[\frac{1}{N}]$ (since $F=\mathbb{Q}%
\cdot\mathcal{R}$). For the next step, we invoke the equivalence of Equation
\ref{lvtx3}, showing that $K_{(p)}(\mathcal{R})$ agrees with $K_{(p)}%
(\widehat{\mathcal{R}}_{p})$ of the completion. But $\widehat{\mathcal{R}}%
_{p}$ decomposes at the product over the completion of $\mathcal{R}$ over all
finite places $v$ over $p$ by Proposition \ref{prop_Subdecompose}, and
combined with Lemma \ref{lemma_topcompare} this yields
\[
K_{(p)}(\mathcal{R})\cong\coprod_{v\mid p}K_{\mathfrak{\tilde{m}}_{v}%
}(\mathcal{R}_{v})\text{,}%
\]
where $v$ runs through the finite places of $F$ lying over the prime number
$p$, and $\mathfrak{\tilde{m}}_{v}$ refers to the unique maximal ideal of
$\mathcal{R}_{v}$. Plugging this into the previous fiber sequence, we arrive
at the fiber sequence%
\begin{equation}
\bigoplus_{v}K_{\mathfrak{\tilde{m}}_{v}}(\mathcal{R}_{v})\longrightarrow
K(\mathcal{R})\longrightarrow K(F)\text{.} \label{lvva1}%
\end{equation}

\begin{remark}
Instead of our somewhat peculiar choice of natural numbers $N$ in Equation
\ref{lcimde7a} we could also have taken any zero-dimensional ideal $I$ in
$\mathcal{R}$, the fiber sequence%
\[
K_{I}(\mathcal{R})\longrightarrow K(\mathcal{R})\longrightarrow
K(\operatorname*{Spec}\mathcal{R}-\operatorname*{Spec}\mathcal{R}/I)
\]
and then taken the filtering colimit over all such ideals partially ordered by
inclusion. This yields the same colimit in Equation \ref{lvva1} (since the
family of ideals $N$ with $N\in\mathbb{Z}_{\geq1}$ is a cofinal subsystem).
\end{remark}

Now consider the long exact sequence of homotopy groups attached to Equation
\ref{lvva1}. Around $\pi_{1}$ and $\pi_{2}$ it specializes to%
\begin{align*}
&  K_{2}(\mathcal{R})\longrightarrow K_{2}(F)\longrightarrow\bigoplus
_{v}K_{\mathfrak{\tilde{m}}_{v},1}(\mathcal{R}_{v})\longrightarrow
K_{1}(\mathcal{R})\overset{\iota}{\longrightarrow}\cdots\\
&  \qquad\qquad\cdots\overset{\iota}{\longrightarrow}K_{1}(F)\longrightarrow
\bigoplus_{v}K_{\mathfrak{\tilde{m}}_{v},0}(\mathcal{R}_{v})\longrightarrow
K_{0}(\mathcal{R})\longrightarrow K_{0}(F)\longrightarrow\cdots\text{.}%
\end{align*}
Note that since we use non-connective $K$-theory, we do not per se know that
the sequence terminates in three $\pi_{0}$-groups, there might well be
negative homotopy groups. This need not concern us however. The middle part at
the arrow $\iota$ is easy to analyze. We recall the definition of $SK_{1}$:%
\[
SK_{1}(R):=\ker\left(  \det:K_{1}(R)\longrightarrow R^{\times}\right)
\text{.}%
\]
The arrow $\iota$ thus induces a commutative diagram%
\[%
\xymatrix{
1 \ar[r] & SK_1(\mathcal{R}) \ar[r] \ar[d] & K_1(\mathcal{R}) \ar
[r]^-{\operatorname{det}} \ar[d] & {\mathcal{R}}^{\times} \ar[r] \ar[d] & 1 \\
1 \ar[r] & SK_1(F) \ar[r] & K_1(F) \ar[r]_-{\operatorname{det}} & {F}^{\times}
\ar[r] & 1.
}%
\]
The horizontal surjections on the right are split. As $F$ is a field, we have
$SK_{1}(F)=1$. We deduce that the sequence is spliced from two exact sequences%
\begin{equation}
0\longrightarrow\mathcal{R}^{\times}\longrightarrow K_{1}(F)\longrightarrow
\bigoplus_{v}K_{\mathfrak{\tilde{m}}_{v},0}(\mathcal{R}_{v})\longrightarrow
K_{0}(\mathcal{R})\longrightarrow K_{0}(F)\longrightarrow\cdots\label{lt1}%
\end{equation}
and%
\begin{equation}
\cdots\longrightarrow K_{2}(\mathcal{R})\longrightarrow K_{2}%
(F)\longrightarrow\bigoplus_{v}K_{\mathfrak{\tilde{m}},1}(\mathcal{R}%
_{v})\longrightarrow SK_{1}(\mathcal{R})\longrightarrow1\text{.} \label{lt2}%
\end{equation}
Although interesting in its own right, we have no immediate use for Equation
\ref{lt1}. Equation \ref{lt2} yields%
\begin{equation}
\cdots\longrightarrow K_{2}(\mathcal{R})\longrightarrow K_{2}(F)\overset
{\partial}{\longrightarrow}\bigoplus_{v}K_{\mathfrak{\tilde{m}},1}%
(\mathcal{R}_{v})\longrightarrow SK_{1}(\mathcal{R})\longrightarrow1\text{.}
\label{lzw3}%
\end{equation}
From Theorem \ref{thm_LocalMain} we know that the boundary map $\partial$ on
each summand $v$ agrees canonically with the full Hilbert symbol at the place
$v$.

Below, we shall use the term \textquotedblleft up to $2$-primary
torsion\textquotedblright. It means that we work in the abelian category%
\begin{equation}
Q:=\frac{\operatorname{AbelianGroups}}{2\text{-}\operatorname*{primary}%
\text{-}\operatorname*{torsion}}\text{.} \label{lcaa1}%
\end{equation}
To set this up in detail note that the $2$-primary torsion abelian groups form
a Serre subcategory in all abelian groups, so the above quotient exists and is
an abelian category itself. A complex of abelian groups becomes exact in this
quotient category if and only if its homology is made of $2$-primary torsion
groups. For example, a map $A\rightarrow B$ is a monomorphism in the quotient
category if and only if its kernel is $2$-primary torsion.

As an alternative characterization, a complex of abelian groups becomes exact
in $Q$ if and only if becomes exact when tensoring it with the flat
$\mathbb{Z}$-module $\mathbb{Z}\left[  \frac{1}{2}\right]  $. In fact, the
category $Q$ is equivalent to the category of $\mathbb{Z}\left[  \frac{1}%
{2}\right]  $-modules.

\begin{theorem}
\label{thm_GlobalMain}Suppose $F$ is any number field. Suppose $\mathcal{R}$
denotes the optimal order. Then, in abelian groups up to $2$-primary torsion,
the diagram%
\[%
\xymatrix{
K_{2}(\mathcal{R}) \ar[r] & K_{2}(F) \ar[r]^-{\partial} \ar@{=}[d] & \bigoplus
_{v \enspace\operatorname{finite}}K_{\mathfrak{\tilde{m}},1}(\mathcal{R}%
_{v}) \ar[d]^{\oplus_v \phi}
\ar[r] & SK_{1}(\mathcal{R}) \ar[r] \ar[d]_{\tau} & 0 \\
& K_{2}(F) \ar[r]_-{h_v} & \bigoplus_{v \enspace\operatorname{noncomplex}}
\mu(F_v) \ar[r]_-{\cdot\frac{m_v}{m}}
& \mu(F) \ar[r] & 0,
}%
\]
commutes, has exact rows, and the downward arrows are isomorphisms. In
particular, the Hilbert reciprocity law gets identified with a sequence coming
directly from a $K$-theory localization sequence (up to $2$-primary torsion).
We deduce%
\[
SK_{1}(\mathcal{R})\cong\mu(F)\text{,}%
\]
up to $2$-primary torsion.
\end{theorem}

Recall that it is known that $SK_{1}(\mathcal{O}_{F})=0$ by the work of
Bass--Milnor--Serre \cite{MR244257}.

\begin{proof}
The exactness of the top row comes from the localization sequence. The bottom
row is the Moore sequence (Theorem \ref{thm_MooreSequence}). We employ Theorem
\ref{thm_LocalMain} for all finite places over odd primes, showing that each
such $\phi$ is an isomorphism. Finite places over $p=2$ have the full (local)
ring of integers as their local optimal order. Thus, at these the boundary map
is just the tame symbol and not the honest Hilbert symbol. But the tame symbol
misses only the $2$-primary torsion summand of $\mu(F_{v})$ as we are over
$p=2$, which disappears after inverting two. Complex places play no role in
the Hilbert reciprocity law. Real places contribute factors $\{\pm1\}$, which
also disappear by inverting two. Thus, even though in principle the sum of the
bottom row should also include real places, it makes no difference to consider
only the finite places here as well. The snake lemma implies the claim about
$SK_{1}$.
\end{proof}

\begin{corollary}
Let $F$ be a number field with $\sqrt{-1}\notin F$. Then the statement of
Hilbert reciprocity up to sign, i.e.%
\[
\prod_{v\text{ }\operatorname*{noncomplex}}h_{v}(\alpha,\beta)^{\frac{m_{v}%
}{m}}=\pm1\qquad\text{for all}\qquad\alpha,\beta\in F^{\times}\text{,}%
\]
can be phrased as the property of being a complex ($d^{2}=0$) for a
localization sequence in $K$-theory.
\end{corollary}

\begin{proof}
Theorem \ref{thm_GlobalMain} shows that the said product is zero in
$\mu(F)\left[  \frac{1}{2}\right]  $, but by our assumption the $2$-torsion
part of $\mu(F)$ is only $\{\pm1\}$.
\end{proof}

One might also ask about the kernel of the leftmost arrow in Moore's sequence.
It is known as the \emph{wild kernel }$WK_{2}$:%
\[
0\longrightarrow WK_{2}(F)\longrightarrow K_{2}(F)\longrightarrow
\bigoplus_{v\text{ }\operatorname*{noncomplex}}\mu(F_{v})\overset{\cdot
\frac{m_{v}}{m}}{\longrightarrow}\mu(F)\longrightarrow0\text{.}%
\]
Following earlier work of Tate, Hutchinson has proven that $WK_{2}%
(F)/K_{2}(F)_{div}\in\{0,\mathbb{Z}/2\}$ along with a precise criterion which
case occurs for what number fields $F$ (both cases occur),
\cite{MR1824144,MR2072396}. Note that our Theorem \ref{thm_GlobalMain} shows
that $K_{2}(\mathcal{R})$ surjects onto the wild kernel after inverting $2$.
This suggests that one can probably also define the global optimal order
analogously to Theorem \ref{thm_LocalMain} as the largest $\mathbb{Z}$-order
in $\mathcal{O}_{F}$ such that $K_{2}(\mathcal{R})\left[  \frac{1}{2}\right]
$ is divisible. This would mirror our characterization of local optimal orders
in a rather neat way.

We shall explain the relation to the picture of Kapranov--Smirnov \cite{ks} in
a future text.

\section{Is this a self-contained proof of Hilbert
reciprocity?\label{sect_SelfCont}}

Given that Gillet's proof of Hilbert reciprocity for function fields uses
nothing more than localization, the reader might (and should) ask to what
extent the method in this text gives a self-contained proof in the number
field case. First, we have freely used the definition of the Hilbert symbol,
which itself uses local class field theory. Using only this, we then could
establish the exactness of%
\[
K_{2}(F)\overset{\partial}{\longrightarrow}\bigoplus_{v\text{ }%
\operatorname*{noncomplex}}\mu(F_{v})\longrightarrow SK_{1}(\mathcal{R}%
)\longrightarrow0
\]
up to inverting $2$. The issues at $p=2$, a familiar nightmare of any number
theorist (and homotopy theorist...) do not appear easy to remove. At places
over $p=2$ Theorem \ref{thm_LocalMain} fails, and the real places do not even
possess a valuation ideal (how to pinch a singularity in the order when there
is no maximal order?). Let us accept these shortcomings. Then the more crucial
issue is that our Hilbert reciprocity law is about the vanishing of terms in
the group $SK_{1}(\mathcal{R})$, which remains elusive without further work
(see Remark \ref{rmk_IdeasRegardingSK1R}).

This is really similar to a completely different approach to Hilbert
reciprocity originating from the papers \cite{MR3954369, clausennc}. These
papers produce, if we only use local class field theory as input, an exact
sequence%
\begin{equation}
K_{2}(F)\overset{\partial}{\longrightarrow}\bigoplus_{v}\mu(F_{v}%
)\longrightarrow K_{2}(\mathsf{LCA}_{F})_{/div}\longrightarrow0\text{,}
\label{lcimde5a}%
\end{equation}
(literally, without having to invert $2$), where $v$ runs over all non-complex
places and $\mathsf{LCA}_{F}$ is the category of locally compact topological
$F$-vector spaces, $G_{/div}$ refers to $G/G_{div}$, the quotient by the
subgroup of divisible elements. The exactness of Equation \ref{lcimde5a} gets
proven loc. cit. in an \textit{entirely different fashion} than the methods of
the present text. There are no singular orders involved. Even though
localization sequences play a role in the proof, one never localizes with
respect to zero-dimensional subschemes. Loc. cit. does \emph{not} realize the
Hilbert symbol as a boundary map. Sequence \ref{lcimde5a} also turns out to be
equivalent to Moore's sequence \cite[Corollary 9.6]{clausennc}, but again with
the issue that the identification $K_{2}(\mathsf{LCA}_{F})_{/div}\cong\mu(F)$
remains elusive without using global class field theory.

\begin{remark}
\label{rmk_IdeasRegardingSK1R}I see several possible approaches to compute
$SK_{1}(\mathcal{R})$. In this paper, in Theorem \ref{thm_GlobalMain} we have
just used\ Moore's sequence. This was easy, but would be circular to be an
independent approach to Hilbert reciprocity. Alternatively, I would believe
one could use the work of Bass--Milnor--Serre on the congruence subgroup
problem \cite{MR244257}. All elements of $SK_{1}(\mathcal{R})$ can be
expressed through Mennicke symbols associated to invertible ideals. As
$\mathcal{R}$ is generally not Dedekind, finitely many ideals will fail to be
invertible, but the others suffice to generate the group. However, the next
step then would be to find relations between these Mennicke symbols, which
would probably require to connect these to power reciprocity symbols as in
Bass--Milnor--Serre and then again rely on properties also equivalent to
already having Hilbert reciprocity available. Another idea might be Iwasawa
theory, because inspecting the corresponding group in Gillet's proof in the
function field case is linked to the Jacobian of the curve after base
extension to the algebraic closure. Iwasawa has taught us that the mixed
characteristic counterpart of this should be $\mathbb{Z}_{p}$-extensions. But
again, relying on Iwasawa theory very quickly necessitates relying on tools
from global class field theory. This remains to be investigated.
\end{remark}

\section{Applications \& Complements}

\subsection{Metaplectic extensions}

Using the above results, one can provide a quick construction of the local and
global metaplectic coverings, including the proof of the reciprocity property
of the global covering vaguely resembling Gillet's proof of Weil reciprocity
(\S \ref{sect_GilletsProof}). We briefly recall the background.

The classical theta function is given by%
\[
\theta(z):=\sum_{n\in\mathbb{Z}}q^{n^{2}}\qquad\text{with}\qquad q:=e^{2\pi
iz}%
\]
for $z$ in the complex upper half-plane. Any even power $\theta^{2m}$ is a
modular form of weight $m$ for $\Gamma_{0}(4)$ and a certain nebentype. The
odd powers therefore were classically understood as something like a
half-integral weight modular form, having a more complicated transformation
behaviour involving the quadratic residue symbol. Weil \cite{MR165033} then
proposed to view these odd powers, and in particular $\theta$ itself, as
automorphic forms on a central extension%
\begin{equation}
1\longrightarrow\mu_{2}\longrightarrow\widehat{G}\longrightarrow
\operatorname*{GL}\nolimits_{2}\longrightarrow1 \label{lwaa0}%
\end{equation}
as opposed to living on the original group $\operatorname*{GL}\nolimits_{2}$.
The appearance of quadratic residue symbols in the transformation formulas
then translates to the defining cocycle $H_{\operatorname*{grp}}%
^{2}(\operatorname*{GL}\nolimits_{2},\mu_{2})$ be given in terms of these.
However, once thinking about such central extensions, one can construct more
complicated ones. Following Kubota \cite{MR0255490}, if $F$ is a number field
and $\mathbf{A}_{F}$ its ad\`{e}le ring, one can construct a central extension%
\begin{equation}
1\longrightarrow\mu(F)\longrightarrow\widehat{G}\longrightarrow
\operatorname*{GL}\nolimits_{2}(\mathbf{A}_{F})\longrightarrow1\text{,}
\label{lwaa1a}%
\end{equation}
which for $F=\mathbb{Q}$ specializes to Equation \ref{lwaa0}. In fact its
restriction to $\operatorname*{SL}\nolimits_{2}(\mathbf{A}_{F})\subset
\operatorname*{GL}\nolimits_{2}(\mathbf{A}_{F})$ is the universal topological
central extension of $\operatorname*{SL}\nolimits_{2}$ (we will not explain
this in detail, but using the topology of the ad\`{e}les, $\operatorname*{SL}%
\nolimits_{2}(\mathbf{A}_{F})$ is naturally a locally compact group and one
can consider the category of central extensions of locally compact topological
groups. Algebraically, the universal central extension is strictly bigger,
\cite{MR244258}). The extension in Equation \ref{lwaa1a} has the following
property: If $F\hookrightarrow\mathbf{A}_{F}$ is the diagonal embedding, the
pullback of the central extension trivializes (i.e. it splits). This property
turns out to be equivalent to Hilbert reciprocity. The pullback of the
extension along the inclusion of a local field $F_{v}\hookrightarrow
\mathbf{A}_{F}$ gives rise to local counterparts of the metaplectic extension.
The theory of metaplectic extensions has branched into several separate lines
of development \cite{MR743816,MR1896177}. We just follow one thread here.

Suppose $R$ is any ring. Then the group of elementary matrices
$\operatorname*{E}(R)\subseteq\operatorname*{SL}(R)$ has a universal central
extension, the \emph{Steinberg group} $\operatorname*{St}(R)$. Its center
agrees with the group $K_{2}(R)$.%
\begin{equation}
1\longrightarrow K_{2}(R)\longrightarrow\operatorname*{St}(R)\longrightarrow
\operatorname*{E}(R)\longrightarrow1\text{.} \label{lwaa1}%
\end{equation}
Thus, for any group homomorphism $\gamma:K_{2}(R)\rightarrow A$ to an abelian
group $A$, one can take the pushout of this central extension along $\gamma$
and get a new central extension of $\operatorname*{E}(R)$. If we describe the
isomorphism class of a central extension of $\operatorname*{E}(R)$ by
$K_{2}(R)$ through its group cocycle in $H_{\operatorname*{grp}}%
^{2}(\operatorname*{E}(R),K_{2}(R))$, this operation just amounts to the
functoriality in coefficients%
\[
\gamma_{\ast}\colon H_{\operatorname*{grp}}^{2}(\operatorname*{E}%
(R),K_{2}(R))\longrightarrow H_{\operatorname*{grp}}^{2}(\operatorname*{E}%
(R),A)\text{.}%
\]
For $F/\mathbb{Q}_{p}$ a finite extension with $p$ odd, take $R:=F$ and
$(\mathcal{R},\mathfrak{\tilde{m}})$ the local optimal order. We may use
Theorem \ref{thm_LocalMain} and the boundary map $\partial\colon
K_{2}(F)\rightarrow K_{\mathfrak{\tilde{m}},1}(\mathcal{R}_{m})$ to get a
central extension%
\begin{equation}
1\longrightarrow\mu(F)\longrightarrow\widehat{G}\longrightarrow
\operatorname*{SL}(F)\longrightarrow1 \label{laa4a}%
\end{equation}
because for fields the inclusion $\operatorname*{E}(F)\subseteq
\operatorname*{SL}(F)$ is the identity. One can now pull this back along
$\operatorname*{SL}\nolimits_{n}(F)\hookrightarrow\operatorname*{SL}(F)$ to
obtain the classical Kubota metaplectic extension for $n=2$. The story for
$\operatorname*{GL}\nolimits_{n}$ is more complicated in general, but one can
get Kubota's extension for $\operatorname*{GL}\nolimits_{2}$ by pulling back
along%
\[
\operatorname*{GL}\nolimits_{2}(F)\hookrightarrow\operatorname*{SL}%
\nolimits_{3}(F)\qquad M\mapsto%
\begin{pmatrix}
M & \\
& \det(M)^{-1}%
\end{pmatrix}
\text{.}%
\]
We can also do all this globally. Suppose $F$ is a number field and
$\mathbf{A}_{F}$ its ad\`{e}le ring. It is easy to set up a map%
\begin{equation}%
\xymatrix{
K_{2}(\mathbf{A}_{F}) \ar[r] & {\left\{  (\alpha_{v})_{v}\in\prod_{v}%
K_{2}(F_{v})\text{ }\left\vert\begin{array}
[c]{l}
\text{with }\alpha_{v}\in\operatorname*{im}K_{2}(\mathcal{O}_{v})\text{ for
all}\\
\text{but finitely many }v
\end{array}
\right.  \right\} } \ar[d]^{\partial} \\
& \bigoplus_{v}K_{\mathfrak{\tilde{m}},1}(\mathcal{R}_{v}) \ar[d] \\
&  SK_{1}(\mathcal{R}),
}
\label{lwaa2}%
\end{equation}
where the first downward arrow uses the boundary map with respect to the local
optimal order and the adelic finiteness condition ensures that the image lands
in a direct sum. The map to $SK_{1}$ then comes from our global formalism.

Now push out Equation \ref{lwaa1} for $R:=\mathbf{A}_{F}$ by the map in
Equation \ref{lwaa2} to obtain a central extension%
\begin{equation}
1\longrightarrow SK_{1}(\mathcal{R})\longrightarrow\widehat{G}\longrightarrow
\operatorname*{SL}(\mathbf{A}_{F})\longrightarrow1\text{.} \label{laa4}%
\end{equation}
By Theorem \ref{thm_GlobalMain} up to a zig-zag of isogenies of $2$-power
orders, this extension agrees with the global metaplectic extension (as an
elaboration: The isomorphism $SK_{1}(\mathcal{R})\left[  \frac{1}{2}\right]
\cong\mu(F)\left[  \frac{1}{2}\right]  $ implies the existence of isogenies of
abelian groups with $2$-power order. Pushouts along these maps on the centers
of the extension provide isogenies between the middle terms $\widehat{G}$ of
the same order). Moreover, this construction trivially has all the properties
as in Kubota's construction.\ Pulling back the extension in Equation
\ref{laa4} along $F_{v}\hookrightarrow\mathbf{A}_{F}$ (for $v$ any finite
place over an odd prime) retrieves the local metaplectic extension of Equation
\ref{laa4a}. The pullback along the diagonal embedding $\iota\colon
F\hookrightarrow\mathbf{A}_{F}$ trivializes.

To see the latter, we only need to understand the isomorphism class of the
extension, i.e. we can work with the group $2$-cocycles. The pullback amounts
to pulling back the cocycle, i.e if $\gamma$ denotes the extension in Equation
\ref{laa4}, we consider%
\[
\iota^{\ast}:H_{\operatorname*{grp}}^{2}(\operatorname*{SL}(\mathbf{A}%
_{F}),SK_{1}(\mathcal{R}))\longrightarrow H_{\operatorname*{grp}}%
^{2}(\operatorname*{SL}(F),SK_{1}(\mathcal{R}))\text{,}%
\]
but the cocycle came from pushing out along a map on the level of $K_{2}$,
which for the pullback $\iota^{\ast}\gamma$ can be factored as%
\[
K_{2}(F)\longrightarrow K_{2}(\mathbf{A}_{F})\overset{\text{Eq. \ref{lwaa2}}%
}{\longrightarrow}SK_{1}(\mathcal{R})\text{,}%
\]
but this is the composition of two successive arrows in the localization
sequence (Equation \ref{lzw3}). Thus, $\iota^{\ast}\gamma$ is the zero
cohomology class. This class corresponds to a trivial central extension. And,
in view of Theorem \ref{thm_GlobalMain} this property is again (up to the
$2$-power isogeny) equivalent to Hilbert reciprocity. Note that, without
worrying about $2$, the central extension by $SK_{1}(\mathcal{R})$ exists
unconditionally and trivializes after pulling back to $\operatorname*{SL}(F)$.
This is some variant of Hilbert reciprocity and of the usual metaplectic covering.

\begin{remark}
[\cite{clausennc}]The group on the upper right in Diagram \ref{lwaa2} can be
shown to be isomorphic to $K_{2}(\mathsf{LCA}_{F})$. The horizontal arrow then
merely amounts to be induced by the functor sending the projective generator
$\mathbf{A}_{F}$ to itself, equipped with its standard locally compact
topology. Moreover, the downward arrow, here denoted by $\partial$, can be
characterized as the cokernel under quotienting out the subgroup of divisible
elements in the source group.
\end{remark}

%

\appendix

\section{$K$-theory in the singular situation\label{sect_appendix_Kthy}}

In this section we summarize some background on Algebraic $K$-theory,
focussing on what we need in this paper. As we will specifically deal with
non-regular rings, we need to use a sufficiently broad framework to handle
such singular situations. We will use the framework of Thomason--Trobaugh
\cite{MR1106918}.

For a scheme\footnote{We tacitly assume that all our schemes are Noetherian
and separated over $\mathbb{Z}$.} we write $K(X)$ for its non-connective
$K$-theory spectrum. This can be defined as the non-connective $K$-theory of
the category of perfect complexes on $X$, or equivalently as the one of vector
bundles on $X$. We mostly deal with rings, so following common practice we
write $K(R):=K(\operatorname*{Spec}R)$ for the $K$-theory of a ring.
Equivalently, this could be defined as the $K$-theory of the category of
finitely generated projective $R$-modules.

If $I$ is an ideal in $R$, we write $K_{I}(R)$ for what would otherwise be
denoted by%
\[
K_{\operatorname*{Spec}R/I}(R)\qquad\text{or}\qquad K(\operatorname*{Spec}%
R\left.  \ \operatorname{on}\ \right.  \operatorname*{Spec}R/I)\text{.}%
\]
This can be defined as the $K$-theory of perfect complexes of $R$-modules
which become acyclic after pulling them back to the open $\operatorname*{Spec}%
R-\operatorname*{Spec}R/I$.

We shall only need a few general facts: First, there is the general
localization theorem, a fiber sequence of spectra%
\begin{equation}
K_{I}(R)\longrightarrow K(R)\longrightarrow K(\operatorname*{Spec}%
R-\operatorname*{Spec}R/I)\text{.} \label{lvtx0}%
\end{equation}
This holds for any commutative unital ring $R$ and any ideal $I$. This is
\cite[Theorem 7.4]{MR1106918}. If $I=(f)$ is a principal ideal, the last term
agrees with $K(\operatorname*{Spec}R[f^{-1}])$. In this case $K_{I}(R)$ can
also be defined as the $K$-theory of perfect complexes $P_{\bullet}$ whose
(exact) basechange $P_{\bullet}\otimes_{R}R[f^{-1}]$ is acyclic:%
\begin{equation}
K_{I}(R)\longrightarrow K(R)\longrightarrow K(R\left[  \frac{1}{f}\right]
)\text{.} \label{lvtx1}%
\end{equation}

We return to the general case. If $I,J$ are coprime ideals, i.e. $I+J=R$, then%
\begin{equation}
K_{IJ}(R)\cong K_{I}(R)\oplus K_{J}(R)\text{.} \label{lvtx2}%
\end{equation}
This is \cite[Corollary 8.1.4]{MR1106918}: As the ideals are coprime, we have
$\operatorname*{Spec}R/I\cap\operatorname*{Spec}R/J=\varnothing$, when
regarded as closed subschemes in $\operatorname*{Spec}R$, so the upper left
corner in the cartesian square loc. cit. is the zero spectrum.

We also need a certain compatibility with completions. Suppose $R$ is a
Noetherian ring and $I$ any ideal. Then there is an equivalence%
\begin{equation}
K_{I}(R)\overset{\sim}{\longrightarrow}K_{I\widehat{R}_{I}}(\widehat{R}%
_{I})\text{,} \label{lvtx3}%
\end{equation}
where $\widehat{R}_{I}$ denotes the $I$-adic completion of $R$, and
$I\widehat{R}_{I}$ is the extended ideal. This is \cite[Proposition 3.19, in
the format of Exercise 3.19.2]{MR1106918}.

Finally, if $R$ is a Noetherian regular ring, then the inclusion of the
category of finitely generated projective $R$-modules into the category of all
finitely generated $R$-modules induces an equivalence%
\[
K(R)\overset{\sim}{\longrightarrow}K(\mathsf{Mod}_{fg}(R))\text{.}%
\]
In this situation we have devissage: Suppose $R$ is a regular Noetherian ring
such that $R/I$ is also regular. Then%
\[
K_{I}(R)\overset{\sim}{\longrightarrow}K(R/I)\text{.}%
\]
Finally, let us note that $K$-theory commutes with finite products of rings
and with filtering colimits.

\section{The Residue Theorem using\ Gillet's method}

Besides wanting to prove Hilbert reciprocity using Gillet's method from
\S \ref{sect_GilletsProof}, one can also prove the residue theorem on curves
using his technique. This is well-known among specialists, but perhaps not too
well recorded in the literature.

\begin{theorem}
[Residue theorem]Let $k$ be a field (of any characteristic). Let $X/k$ be a
geometrically integral smooth proper curve with function field $F:=k(X)$. Then
the composition
\[
\Omega_{F/k}^{1}\overset{\partial}{\longrightarrow}\bigoplus_{v}H_{v}%
^{1}(X,\Omega_{X/k}^{1})\overset{\operatorname*{Tr}_{\{v\}/k}}{\longrightarrow
}k
\]
is zero, where $\partial$ denotes the residue of a rational $1$-form. Here
$H_{v}^{1}$ denotes coherent cohomology with support\footnote{The right
derived functors of the functor of global sections with support in the given
closed point, $R\Gamma_{v}$.} in the closed point belonging to the place $v$.
\end{theorem}

\begin{proof}
As in Gillet's proof, we use the localization sequence, but this time for
Hochschild homology over $k$. To this end note that Hochschild homology is
also a localizing invariant in the sense of Blumberg--Gepner--Tabuada
\cite{MR3070515}. We get the fiber sequence%
\begin{equation}
HH_{Z}(X)\longrightarrow HH(X)\longrightarrow HH(X-Z)\text{.}%
\end{equation}
As before, we may split $Z$ into its disjoint connected components and get
$HH_{Z}(X)=\bigoplus_{P}HH_{P}(X)$, where $P$ runs through the finitely many
points of $Z$. Finally, by the Hochschild--Kostant--Rosenberg (HKR)
isomorphism we have $HH_{n}(X)\cong\Omega_{X/k}^{n}$. As $\dim X=1$, we only
get a non-trivial statement for $n=1$. For unravelling the Hochschild homology
with support in a closed point and showing that the boundary map agrees with
the residue, see the HKR\ theorem with support from \cite[Prop. 2.0.1]%
{MR3990809}, noting that Hochschild homology with support is denoted by
$H^{Z}$ loc. cit.
\end{proof}

One can also use cyclic homology over $k$. This also proves the residue theorem.

\bibliographystyle{amsalpha}
\bibliography{ollinewbib}

\end{document}